\documentclass[a4paper,reqno, 11pt]{amsart}

\usepackage[utf8x]{inputenc}  

\usepackage{fullpage}
\usepackage{srcltx} \usepackage{pdfsync}
\usepackage{calc,amsfonts,amsthm,amscd,epsfig,psfrag,amsmath,amssymb,enumerate,graphicx,mathrsfs}
\usepackage{wrapfig} 
\usepackage{caption}
\usepackage[usenames]{color}\usepackage{graphicx}

\usepackage[showlabels,sections,floats,textmath,displaymath]{}

\numberwithin{equation}{section}

\DeclareMathSymbol{\leqslant}{\mathalpha}{AMSa}{"36} 
\DeclareMathSymbol{\geqslant}{\mathalpha}{AMSa}{"3E} 
\DeclareMathSymbol{\eset}{\mathalpha}{AMSb}{"3F}     
\renewcommand{\leq}{\;\leqslant\;}                   
\renewcommand{\geq}{\;\geqslant\;}                   

\newcommand{\abs}[1]{ \lvert #1 \rvert}
\newcommand{\norm}[1]{ \lVert #1 \rVert}
\newcommand{\td}[0]{\mathrm{d}}

\theoremstyle{plain}{ 
\newtheorem{theorem}{Theorem}[section]

\newtheorem{lemma}[theorem]{Lemma}
\newtheorem{proposition}[theorem]{Proposition}
\newtheorem{corollary}[theorem]{Corollary}

\newtheorem{definition}[theorem]{Definition}
}

\theoremstyle{definition}{
\newtheorem{remark}[theorem]{Remark}

}

\theoremstyle{remark}{
\newtheorem{notation}[theorem]{Notation}
}

\newcommand{\cC}{\ensuremath{\mathcal C}}

\newcommand{\cF}{\ensuremath{\mathcal F}}

\newcommand{\cH}{\ensuremath{\mathcal H}}

\newcommand{\cL}{\ensuremath{\mathcal L}}

\newcommand{\cP}{\ensuremath{\mathcal P}}

\newcommand{\cT}{\ensuremath{\mathcal T}}

\newcommand{\bbC}{{\ensuremath{\mathbb C}} }

\newcommand{\bbE}{{\ensuremath{\mathbb E}} }

\newcommand{\bbN}{{\ensuremath{\mathbb N}} }

\newcommand{\bbP}{{\ensuremath{\mathbb P}} }
\newcommand{\bbQ}{{\ensuremath{\mathbb Q}} }
\newcommand{\bbR}{{\ensuremath{\mathbb R}} }
\newcommand{\bbS}{{\ensuremath{\mathbb S}} }

\newcommand{\bbU}{{\ensuremath{\mathbb U}} }

\newcommand{\bbZ}{{\ensuremath{\mathbb Z}} }

\newcommand{\E}{\mathbb{E}}

\title{Central limit theorem for T-graphs}

\author{Beno\^it  Laslier}
\address{
 Institut Camille Jordan, Universit\'e Lyon 1, 43
  bd du 11 novembre 1918, 69622 Villeurbanne, France\\
E-mail: laslier@math.univ-lyon1.fr}

\date{}

\begin{document}

\begin{abstract}
  In this paper, we establish a 
quenched invariance principle for the random
  walk on a certain class of infinite, aperiodic, oriented
 random planar
  graphs called ``T-graphs'' \cite{Kenyon2007b}. These graphs appear,
  together with the corresponding random walk, in a work
  \cite{Kenyon2007} 
  about the lozenge tiling model, where
  they are used to compute correlations between lozenges inside large finite
  domains. The random walk in question is balanced, i.e. it is
  automatically a martingale.  

Our main ideas are inspired by the proof of a
  quenched central limit theorem in stationary ergodic environment on $\bbZ^2$ \cite{Lawler1982, Sznitman2002}.
  This is somewhat surprising, since the environment 
  is neither defined on $\bbZ^2$ nor really random: the graph is instead quasi-periodic and all the
  randomness is encoded in a single random variable $\lambda$ that is
  uniform in the unit circle.
  We prove that the covariance matrix of the
  limiting Brownian Motion is proportional to the identity, despite
  the fact that the graph does not have obvious symmetry
  properties. This covariance is identified using the knowledge
  of a specific discrete harmonic function on the graph, which is provided by the link
  with lozenge tilings.
\end{abstract}

\maketitle

\section{Introduction}\label{sec:intro}



In \cite{Kenyon2007} and in \cite{Kenyon2007b}, a class of aperiodic
graphs called $T$-graphs was introduced and some deep links between
these graphs and both the uniform spanning tree and the dimer or
perfect matching model were proved. In particular in \cite{Kenyon2007}
they appear as a tool to express correlation functions in the
hexagonal dimer model (also known as the lozenge tiling model) see
Section \ref{sec:dimers} for
details. 
Using this method, one can relate the large scale behavior of dimers,
which is one of the main questions on dimer models, to the way
discrete harmonic functions approximate a continuous harmonic limit.
Here we give a further step in this direction by proving a central
limit theorem for the random walk on these graphs, which shows that
discrete harmonic functions do indeed resemble continuous harmonic
functions on large scale. This is not trivial because T-graphs
generically do not have any exact symmetry. We became aware while
finishing the writing of this paper of another work on harmonic
functions on T-graph \cite{Li2013}. Their methods are very different
from ours since they do not use a central limit theorem. They work
with more general graphs but only obtain convergence of discrete
functions to their continuous counterpart along sub-sequences.

\begin{figure}
  \includegraphics[width=0.5\textwidth]{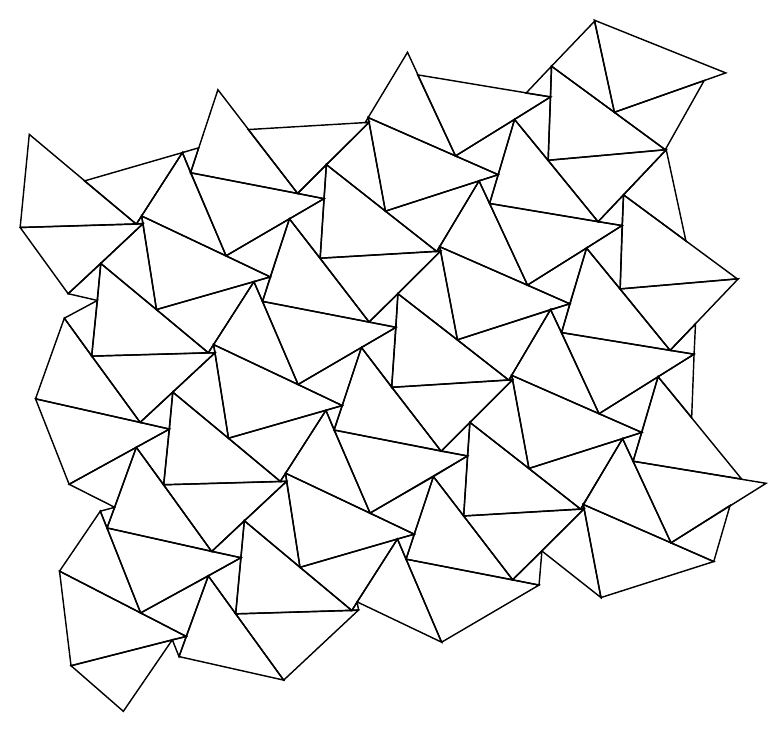}
\caption{A picture of a finite domain inside a T-graph showing clearly the properties of \ref{prop:prop_geometrique}. Image taken from \cite{Kenyon2007}.}
\label{fig:image_globale}
\end{figure}

Unfortunately the method we use does not provide any speed of
convergence so we do not get any accurate estimate on the difference
between continuous and discrete harmonic functions.
However the result is valuable on its own because 
our environment is very far from being IID: all the randomness in the T-graph is encoded in a uniform variable $\lambda $ on the unit circle, and conditionally on $\lambda$ the graph is deterministic and quasi-periodic.
In this framework interesting mathematical
challenges arise. Increments of the random walk are highly correlated and
some of the important concepts used with random environment, like
renewal times, cannot be used. Furthermore the definition of the graph
itself is quite involved so even simple facts like connectedness are
non trivial. One can refer to \cite{Zeitouni2002} for a general overview of random walk in random environment.

Keeping these difficulties in mind, it is striking to see that the ideas of
\cite{Lawler1982, Sznitman2002} carry on; the proof is thus also a
testimony of the robustness of the method. In particular an important
point to note is the role of ergodicity of the graph (i.e. the
environment) with respect to translations. Usually one looks at
ergodicity with respect to some group.
Here on the other hand, the translations that send a vertex to another do
 \emph{not} form a group (or any usual algebraic structure)
so one might think that we cannot use ergodic theory. However we do
not need any structure on translations to define ergodicity in the
sense that any translation invariant event must have probability $0$
or $1$. As it will appear later (see Remark \ref{rq:ergodicite}) this will give enough information on
the (spatial) environment to prove that trajectories of the random walk
are ergodic with respect to time shifts (which do form a semi-group)
and to use Birkhoff ergodic theorem. This remark might be useful to
study the random walk on other kind of environments where translations
do not form a group, like random graphs.

The rest of the paper is organised as follows. We first give the
construction of the graphs we are interested in (section
\ref{sec:construction}) and derive some useful properties
(section \ref{sec:geometrie}). The random walk we will use is then defined in section \ref{sec:marche_aleatoire}, where we also state the main theorem on quenched invariance principle and convergence of discrete harmonic function. The proof of the theorem
is divided into two independent parts. In section \ref{sec:TCL}
we use ergodicity arguments and the martingale invariance principle
to prove 
almost-sure convergence of the random walk to a brownian motion with
an unknown deterministic covariance. Finally in section
\ref{sec:covariance} we see that the covariance has to be proportional
to the identity, 
using the a priori knowledge of a harmonic function provided by
the mapping with random lozenge tilings \cite{Kenyon2006,Kenyon2007b}.

\section{T-graph construction}\label{sec:construction_hex}

In this section, we construct the family of graphs and the random walk we will study in the later sections. The specific structure of the graphs produced by the construction will be of key importance in section \ref{sec:topologie}.

\subsection{Hexagonal lattice}\label{sec:coordonnees}


\begin{wrapfigure}{R}{6.5cm}
\begin{center}
  \includegraphics[width=5cm]{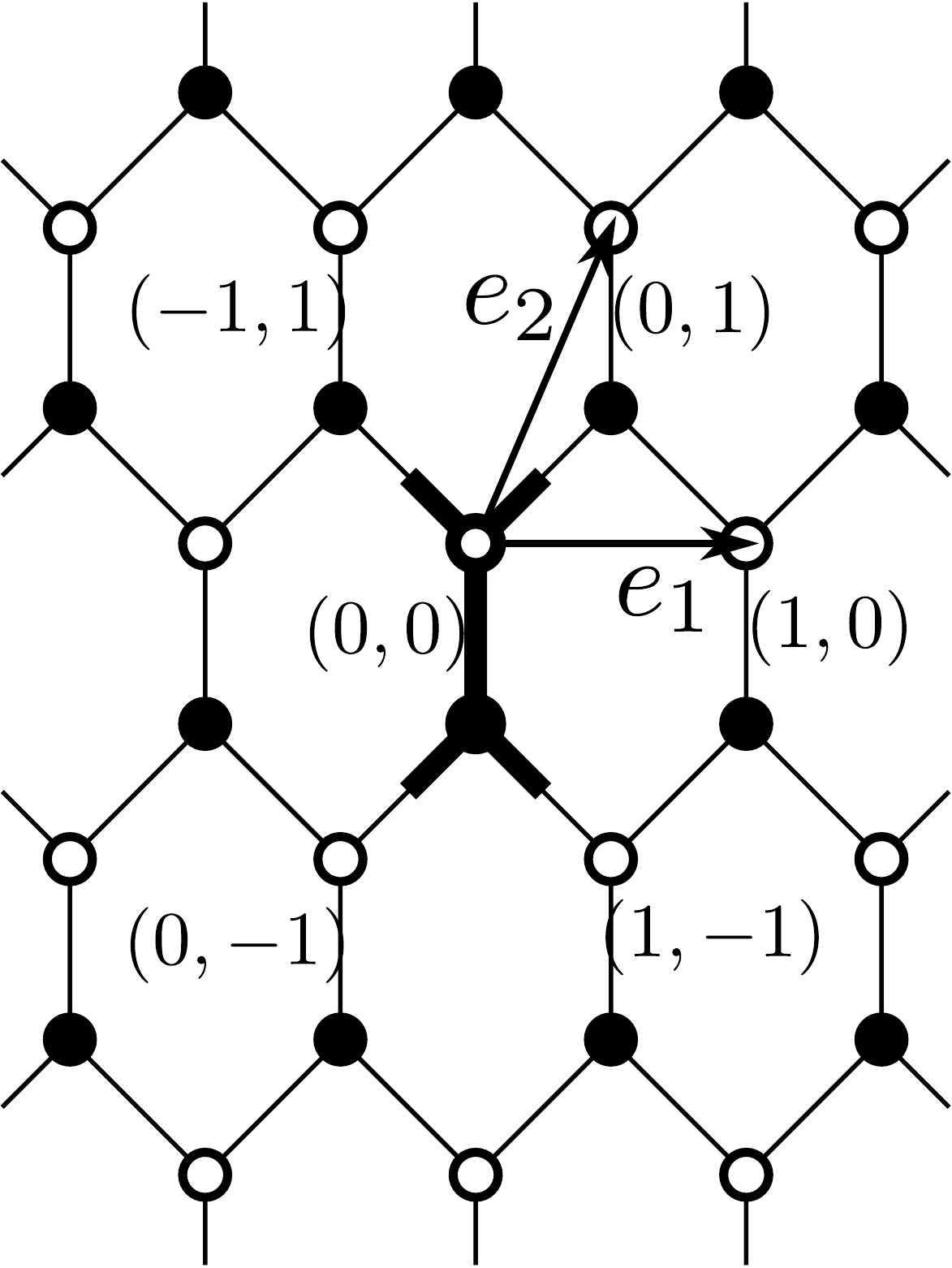}
\captionsetup{width=6cm}
\caption{An illustration of the coordinates we use on the hexagonal lattice. Near each vertical edge are indicated the (common) coordinates of its two endpoints.}
\label{fig:coordonees}
\end{center}
\vspace{-1cm}
\end{wrapfigure} 

First of all we define suitable coordinates on the infinite hexagonal
lattice. Of course the specific choice we give here plays no essential
role so this section is only about fixing notations. However it is
still quite important in practice because we will use several explicit
formulas that depend on the choice of coordinates.

\begin{notation}\label{def:coordonees}
  We embed the hexagonal lattice $\cH$ in the plane in the way
  represented in figure \ref{fig:coordonees}. We call fundamental domain and write
  $\cH_1$ the two vertices with thicker lines. We let $e_1$ and $e_2$
  be the two vectors represented. Given $v$ a vertex of $\cH$, we call
  coordinates of $v$ the unique $n,m$ such that $v-me_1-ne_2 \in
  \cH_1$. Note that given $m,n$ there are exactly two vertices
  with coordinates $(m,n)$, the top one is called a white vertex, the
  bottom one is called black. 

  We will write $m(v)$ and $n(v)$ for the coordinates of the vertex
  $v$. We will also write $b(m,n)$ and $w(m,n)$ for the black and
  white vertices of coordinates $(m,n)$.
\end{notation}

\begin{remark}\label{rq:valeur_fg}
  The three neighbors of a point $b(m,n)$ are $w(m,n)$, $w(m,n-1)$ and
  $w(m+1,n-1)$ while the three neighbors of $w(m,n)$ are $b(m,n)$,
  $b(m,n+1)$ and $b(m-1,n+1)$. We will call edges $w(m,n)b(m,n)$
  \emph{vertical}, edges $w(m,n)b(m,n+1)$ \emph{north east-south west}
  (NE-SW) and edges $w(m,n)b(m-1,n+1)$ \emph{north west-south east}
  (NW-SE).
\end{remark}

\begin{notation}
  We write $\cH^*$ for the dual graph of $\cH$. This is a triangular 
lattice. Each of its faces contains a vertex of $\cH$ and it is 
called black/white according to the color of that vertex. Vertices of $\cH^*$ 
can be associated to the point in the centre of a face of $\cH$.
For a vertex $v$ of
  $\cH^*$ we let $(m(v),n(v))$ be the (common) coordinates of the two
  points just right of $v$.
\end{notation}

\subsection{Construction}\label{sec:construction}


T-graphs are defined by integration of an explicit 1-form on the edges
of $\cH^*$. In this section we define this form and verify that its
primitive is well defined.

\begin{notation}
  Let $\lambda$ be a complex number of modulus one and let $\Delta$ be
  a triangle of area one. We let $a\alpha$, $b\beta$ and $c\gamma$ be
  the complex numbers corresponding to its sides, taken in the
  counterclockwise order, with $a,b,c$ real positive and $\alpha,
  \beta, \gamma$ complex of modulus one. These parameters will be
  fixed for the rest of the section and will thus be often omitted
  from the notations.
\end{notation}

The role of $\lambda$ will be clarified in Proposition \ref{prop:translation}
and in Section \ref{sec:topologie}.

\begin{notation}
We define the following functions :
\begin{itemize}
  \item $f$ is defined on white vertices by $f(w(m,n)) = (\frac{\beta}{\gamma})^{m} (\frac{\beta}{\alpha})^{n}$
  \item $g$ is defined on black vertices by $g(b(m,n)) = \alpha (\frac{\beta}{\gamma})^{m} (\frac{\beta}{\alpha})^{n}$
  \item $K(w,b)$ on edges defined by $K(w,b)=a$ on vertical edges, $b$ on SW-NE edges and $c$ on NW-SE edges.
\end{itemize}
Remark that the three functions depend on $\Delta$ and not on
$\lambda$. We will write $f_\Delta$, $g_\Delta$ and $K_\Delta$ if we
want to emphasize this dependence.
\end{notation}

\begin{remark}
  $f$ and $g$ are defined in order to have $\bar f(w) g(b)$ equals to $\alpha$ (resp. $\beta$, $\gamma$) when $w$ and $b$ are the endpoints of a vertical (resp. NE-SW, NW-SE) edge.
\end{remark}

\begin{proposition}\label{prop:divergence_nulle}
We have :
  \begin{itemize}
    \item for any black vertex $b$, $\sum_{w \sim b} f(w) K(w,b) = 0$
    \item for any white vertex $w$, $\sum_{b \sim w} K(w,b) g(b) =0$,
  \end{itemize}
\end{proposition}
where $w\sim b$ means that $w$ and $b$ are neighbouring vertices.
\begin{proof}
  The three terms of the sums are, up to a multiplicative constant,
  the edge vectors of $\Delta$ so they sum to $0$.
\end{proof}

\begin{notation}
We let $\phi_{\lambda\Delta}$ denote the following flow on oriented edges:
\[
  \phi(wb) = K(w,b) \Re(\bar \lambda \bar f(w)) \lambda g(b) 
\]
and $\phi(bw) = -\phi(wb)$, with $\Re(z)$ the real part of $z$. We let
$\phi^*$ denote the dual flow on oriented edges of $\cH^*$ obtained by
rotating $\phi$ by $+\pi/2$ (counter-clockwise). That is, crossing the edge $bw$
with the white vertex on the left gives flow $+ K(w,b) \Re(\bar
\lambda \bar f(w)) \lambda g(b) $

By Proposition \ref{prop:divergence_nulle}, $\phi$ has zero divergence and thus
the flow of $\phi^*$ around any face of $\cH^*$ is zero. This implies the existence of a function $\psi_{\lambda \Delta }$ on $\cH^*$, unique up to a constant, such that
$\forall v, v' \in \cH^*$, $\psi(v') - \psi(v) = \phi^*(vv')$. We fix
the constant by setting $\psi=0$ on $v(0,0)$ the vertex of $\cH^*$ on the left
of the fundamental domain.

We extend $\psi$ linearly to the edges of $\cH^*$, so that $\psi$ maps
$\cH^*$ to a subset of $\bbC$. We define $T_{\lambda\Delta} =
\psi_{\lambda\Delta }(\cH^*)$. $T$ is the ``T-graph'' we are
interested in.  We will see in Proposition \ref{prop:translation} that the graph so
obtained is ``quasi-periodic'' (or periodic if the angles of the triangle $\Delta$
are rational multiples of $\pi$).
\end{notation}

\begin{remark}
  The definition of $\phi$ might seem strange, especially taking the
  complex conjugate inside a real part. However we claim that this is
  the natural definition. Indeed if we remove the real part from the
  definition of $\phi$ we get a flow $\tilde \phi$ with $\tilde \phi =
  a \alpha$ (resp. $b\beta$,$c\gamma$) on vertical (resp. NW-SW,
  NW-SE) edges. Its primitive $\tilde \psi$ is the linear mapping
  from $\cH^*$ to $\bbC$ that makes all triangles similar to
  $\Delta$. In a sense with the real part we get a perturbation of
  this linear map where all black faces are flattened to segments. This will be
  made more clear in the next section.
\end{remark}

\begin{remark}
  We have not specified how to choose which side of $\Delta$ is
  $a\alpha$ and corresponds to vertical edges, so it may seem that
  there is an ambiguity in the construction. However this is not the
  case because $T$ does not depend on this choice. Indeed, making
  a different choice is equivalent to rotating the hexagonal lattice
  $\cH$ by $2\pi/3$ (which leaves $\cH$ invariant) so the new function
  $\tilde \psi$ verifies $\tilde\psi(x) = \psi( e^{2i\pi/3}x)$ and we
  see that $\tilde T=\tilde\psi( \cH^*)=\psi( \cH^*)=T$.
\end{remark}

\subsection{Geometric properties of T-graphs}\label{sec:geometrie}

These properties are given both
because they enter the proof of the central limit theorem, and also
because they allow to visualize the type of graphs we are working
with. The results and some of the proofs are taken from \cite{Kenyon2007} and \cite{Kenyon2007b} (in the latter a more general class of graphs is studied but full details
are only given in the case of a finite graph). We also give explicit
formulas when possible.

\begin{proposition}
  $\psi$ is almost linear, more precisely if $\ell(m,n)= \frac{a\alpha}{2}m - \frac{c \gamma}{2} n$, then $\psi(v) - \ell(m(v),n(v))$ is bounded.
\end{proposition}

\begin{proof}
  This follows from direct computation. Let $v$ a vertex of $\cH^*$: it is on the left of the vertical edge of coordinates $(m,n)$ so, assuming for simplicity both coordinates are positive, 
\begin{align*}
  \psi(v) & = \sum_{j=0}^{m-1} \phi^*( w(0,j)b(0,j)) + \sum_{j=0}^{n-1} \phi^*( w(m,j)b(m-1,j+1) ) \\
	  & = \sum_{j=0}^{m-1} (+ a) \frac{1}{2}\left(\bar \lambda \left(\frac{\beta}{\gamma}\right)^{-j} + \lambda \left(\frac{\beta}{\gamma}\right)^j \right) \alpha\lambda \left(\frac{\beta}{\gamma}\right)^j \\
	  &  \quad
	+ \sum_{j=0}^{n-1} (-c) \left(
			 \bar \lambda \left(\frac{\beta}{\gamma}\right)^{-m}
				       \left(\frac{\beta}{\alpha}\right)^{-j}
			+
			\lambda \left(\frac{\beta}{\gamma}\right)^{m} 
				  \left(\frac{\beta}{\alpha}\right)^{j}
			\right) 
	    \lambda\alpha \left( \frac{\beta}{\gamma}\right)^{m-1} 
			   \left(\frac{\beta}{\alpha}\right)^{j+1}.
\end{align*}
In each sum there are two terms. In the first the powers of $\alpha$,
$\beta$, $\gamma$ cancel out so they give a linear contribution, in
the second they add up so these terms oscillate and give bounded
geometric sums. Finally
\[
  \psi(v) = +\frac{a\alpha}{2}m - \frac{c \gamma}{2} n + \text{bounded}.
\]
When the coordinates are not positive the same computation holds and we just have to change the signs.
\end{proof}

\begin{remark}
  Since $\alpha$ and $\gamma$ are not collinear, the linear application $\ell$ is non degenerate and
  $T$ covers the whole plane: any point of $\bbC$ is at
  bounded distance from $T$.
\end{remark}

The following result shows that $T$ is quasi-periodic, in the sense that translations of the graphs have properties similar to iterates of an irrational rotation, or periodic if the 
ratios $\beta/\alpha$ and $\beta/\gamma$ are both roots of the unity. Also, it clarifies
the role of $\lambda$ in the 
  construction:  a translation of the graph is
  equivalent to a change of $\lambda$.

\begin{figure}
  \includegraphics[width = 0.7\textwidth]{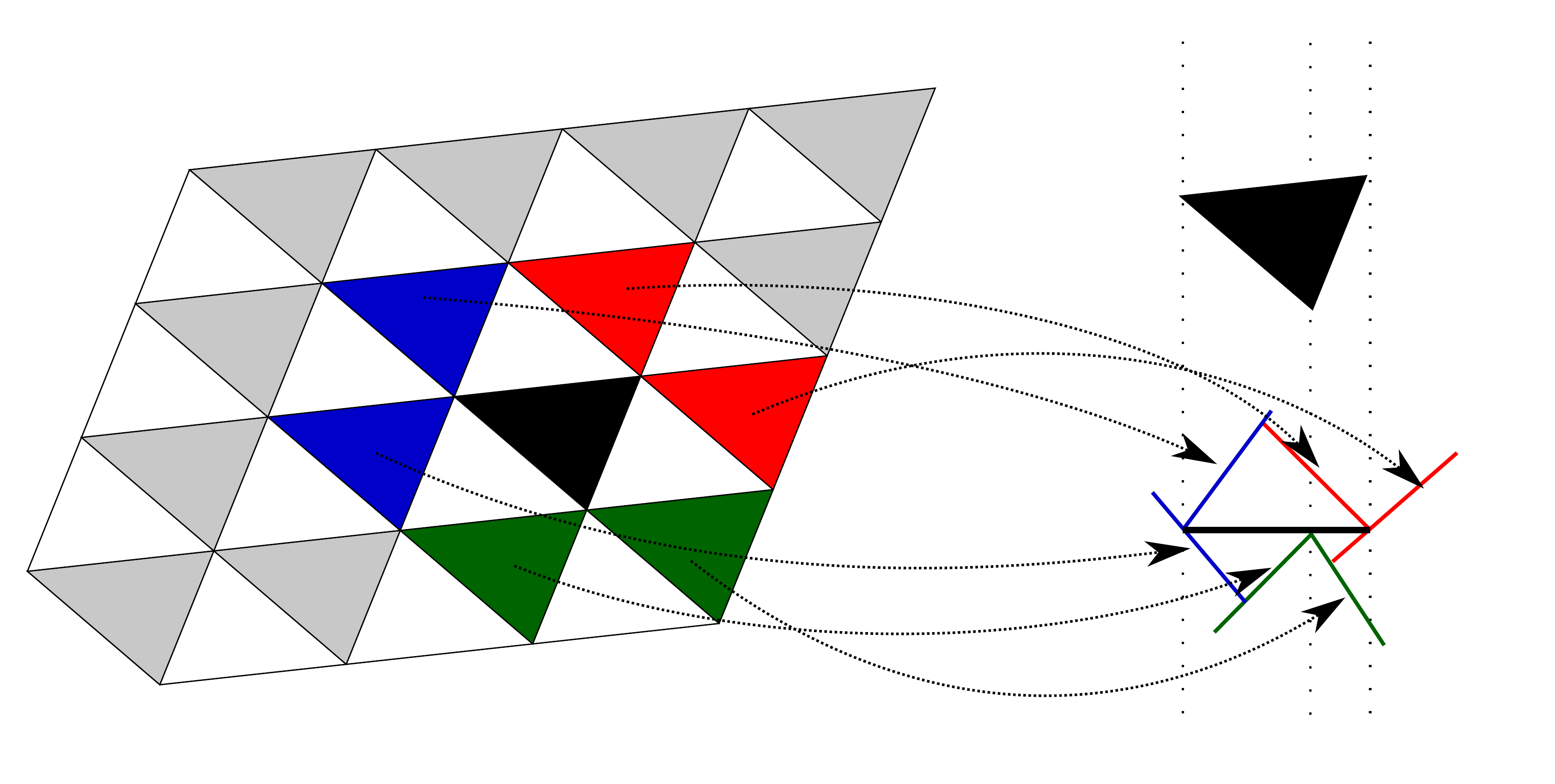}
\caption{A schematic view of the image of a black face and its neighbours in $\cH^*$. }
\label{fig:image_locale}
\end{figure}

\begin{proposition}\label{prop:translation}
  Let $T_{\lambda\Delta}$ the graph constructed above (recall that we
  choose $\psi=0$ on the vertex of $\cH^*$ just left of the fundamental
  domain $\cH_1$). Let $v $ be 
the
 vertex of $\cH^*$ of coordinates $(m,n)$ and let
  $T'$ be the graph constructed in the same way but taking
  $\psi(v)=0$. Then we have $T' = T_{\lambda'\Delta}$ with $\lambda'=
  \lambda (\frac{\beta}{\gamma})^{m} (\frac{\beta}{\alpha})^{n}$.
\end{proposition}
\begin{proof}
  This is immediate from the definition : the change of $\lambda$ is equivalent to multiplying both $f$ and $g$ by $(\frac{\beta}{\gamma})^{m} (\frac{\beta}{\alpha})^{n}$ which in turn corresponds to a translation of the origin.
\end{proof}

Here are collected some geometric facts about $T$:
\begin{proposition}\label{prop:prop_geometrique} 
  $T$ has the following properties (see Figure \ref{fig:image_globale}):
\begin{enumerate}
\item \label{prop:prop_geometrique:face_noire} The image of the black
  face of $\cH^*$ containing the vertex $b$ of $\cH$ is a segment.
More precisely, it is a suitable translation of  $-\lambda g(b) \Re(
  \bar \lambda \bar g(b)  \Delta)$ (we use the convention that, for a subset $U\subset \bbC$, $\Re(U)=\{\Re(z),z\in U \}$);
\item \label{prop:prop_geometrique:face_blanche} The image of the
  white face containing the vertex $w$ is a triangle similar to
  $\Delta$ and with the same orientation. More precisely, it is a
  translation of $\lambda f(w) \Re(\bar \lambda \bar f(w) )
  \Delta$. Remark that multiplication by $ \lambda f(w)$ is a rotation
  and multiplication by $\Re(\bar \lambda \bar f(w) ) $ is a
  contraction;
\item \label{prop:prop_geometrique:non_degenere} The length of
  segments is uniformly bounded away from zero independently of
  $\lambda$; for generic $\lambda$ no triangle is degenerate to a
  point ;
\item \label{prop:prop_geometrique:endpoints} For any $\lambda$ (resp. generic $\lambda$),
    for any vertex $v$ of $\cH^*$, $\psi(v)$ belongs to at least (resp. exactly) three
    segments: generically it is an endpoint of two of them and in the interior of
    the third one. All endpoints of segments are of the above form
    $\psi(v)$ with $v$ a vertex of $\cH^*$;
 \item \label{prop:prop_geometrique:non_recouvrement} The triangular images of
white faces
    cover the plane and do not intersect, that is any $x$ not in a 
    segment belongs to a unique face of the T-graph;
  \item \label{prop:prop_geometrique:non_intersection} Segments
    do not intersect in their interior.
 \end{enumerate}
\end{proposition}


\begin{proof} 
  Points (\ref{prop:prop_geometrique:face_noire}),
  (\ref{prop:prop_geometrique:face_blanche}),
  (\ref{prop:prop_geometrique:non_degenere}) come directly from the
  construction. As an example we give the computation that proves
  (\ref{prop:prop_geometrique:face_blanche}).  Let $w$ be a white vertex
  of coordinates $m,n$. Let $v_1,v_2,v_3$ the vertices of $\cH^*$
  around $w$, taken in the counterclockwise order starting from the
  lower-left one. We have
\[
\psi(v_2) = \psi(v_1) + K(w,b(m,n)) \Re(\bar \lambda \bar f(w) ) \lambda g(b(m,n)) = \psi(v_1) + a \alpha \Re( \bar\lambda \bar f(w)) \lambda f(w) 
\]
 and similarly 
\[
  \psi(v_3) = \psi(v_2) + b \beta \Re( \bar \lambda \bar f(w)) \lambda f(w) .
\]
We see that the image of the white face around $w$ is equal to $\Delta$ rotated by $\lambda f(w)$ and scaled by $\Re( \bar \lambda \bar f(w))$.

For point (\ref{prop:prop_geometrique:endpoints}) it is immediate from
the construction that $\psi(u)$ is in at least three segments and
generically in the interior of one of them (just look at the three
segments corresponding to the black faces of $\cH^*$ around $u$). 
Let $S$ be this segment and let $b$ be the corresponding black face.
It is easy to check from the formulas that the image of the three white faces
neighboring $u$ cover on side of $S$, while the other side is covered by the image of the 
the white face neighboring $b$ and not $u$ (see figure \ref{fig:image_locale}).
By point (\ref{prop:prop_geometrique:non_recouvrement}), which does not requires this part of of point (\ref{prop:prop_geometrique:endpoints}), no other face cover a neighborhood of $\psi(u)$ and thus $\psi(u)$ is in no other segment.

We turn to point (\ref{prop:prop_geometrique:non_recouvrement}) from which point
(\ref{prop:prop_geometrique:non_intersection}) follows easily.  The key idea
of the proof is to combine almost linearity of $\psi$ with the fact
that all faces have the same orientation and to look at winding
numbers, see figure \ref{fig:pas_recouvrement} for an illustration.  Let $x \in \bbC$ and suppose by
contradiction that $x$ is in the interior of two faces
$\psi(w_1)$ and $\psi(w_2)$ (writing here, with some abuse of
notation, $w_1$ and $w_2$ for white faces of $\cH^*$). Let $\cC$
denote the (possibly non-connected) closed curve going once around
$w_1$ and once around $w_2$ anticlockwise: $\psi(\cC)$ has winding number $+2$ around $x$. Consider a simple closed curve $\cL$ in $\cH^*$ going
around both $w_1$ and $w_2$ in anticlockwise order. By point
(\ref{prop:prop_geometrique:face_blanche}), the image of a white face
of $\cH^*$ can only have winding number $0$ or $+1$ around a point. The winding number of $\psi(\cL)$ around $x$ is the sum of the winding numbers around $x$ of all the white faces inside $\cL$ so it is at least the winding number of $\psi(\cC)$, i.e.  at least  $+2$. 
On the other hand, suppose that the loop $\cL$ is  very large and that  both $w_1$ and $w_2$ are very far from it (but still inside it): then, 
by almost linearity of $\psi$, $\psi(\cL)$ has winding 
number $1$ around $x$, which leads to a contradiction.

\begin{figure}
\includegraphics[width=0.7\textwidth]{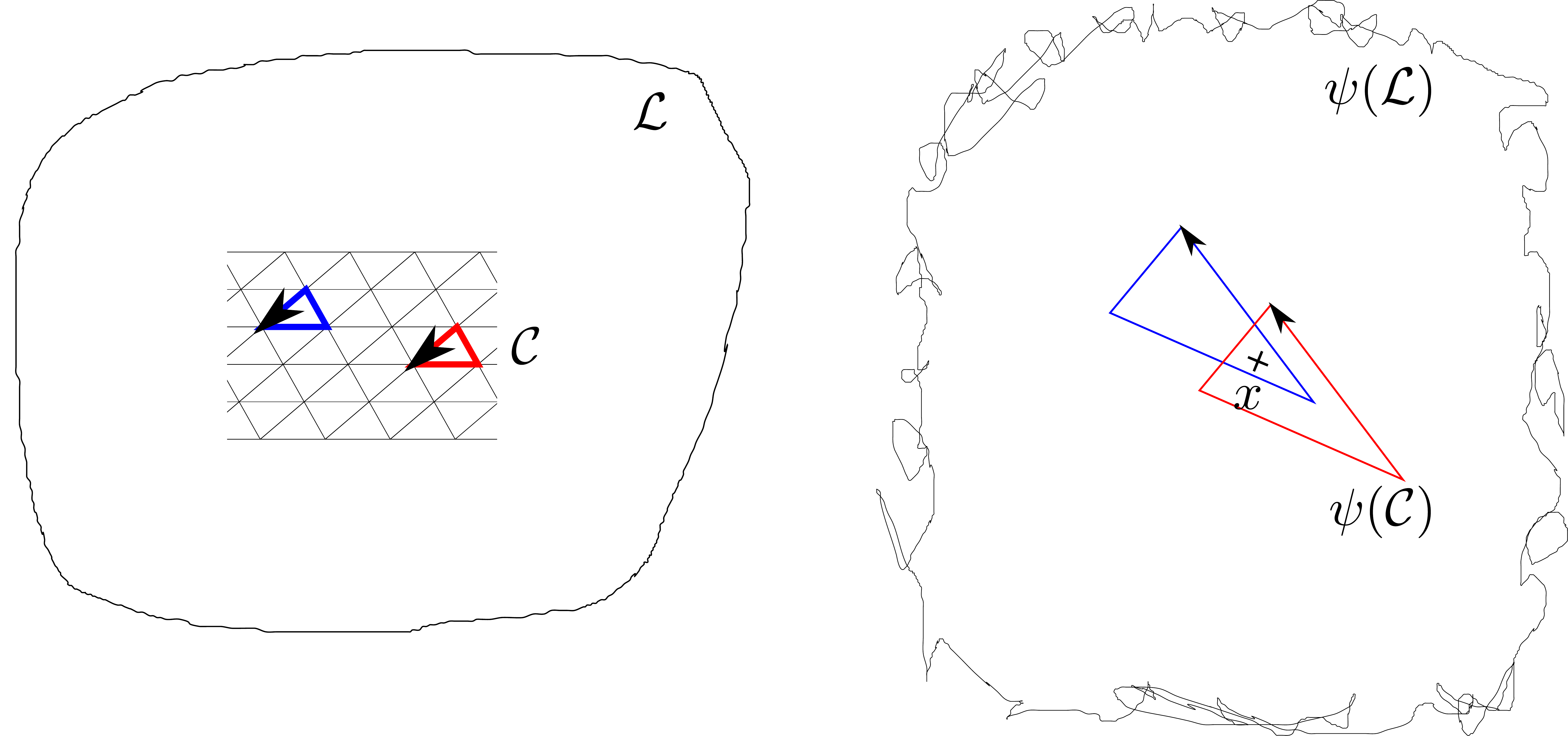}
\caption{An illustration of the proof that faces of $T$ do not overlap. The curve $\psi(\cC)$ made by the two triangles have winding number $2$ around $x$ while the large curve $\psi(\cL)$ have winding number $1$.}
\label{fig:pas_recouvrement}
\end{figure} 

Finally let $\Psi$ be the union of all closed faces, we need to check
that $\Psi = \bbC$. We already said for point (\ref{prop:prop_geometrique:endpoints}) that any
segment is adjacent to three faces, with two on one side and the last
one in the other. Thus segments are never on the boundary of $\Psi$ so
$\Psi$ has no boundary and since it is not empty 
and every point of $\bbC$ is at finite distance from $T$, we have $\Psi =
\bbC$.
\end{proof}

\begin{definition}\label{def:non_degenere}
  The image of a white face of $\cH^*$ is said to be degenerate if its
  size is zero. We will call such a point a degenerate face. A segment
  is said to be degenerate if it has no vertex in its interior. A
  T-graph is said to be degenerate if any of its segment or face is
  degenerate. We will say that a face (resp. a segment) is $\epsilon$
  almost degenerate if its area is smaller than $\epsilon^2$ (resp. has
  its interior point at distance less than $\epsilon$ from its
  endpoints).
Given an almost degenerate edge, we will call the sub-segment connecting 
the two vertices at distance at most $\epsilon$ the ``short sub-segment''.
\end{definition}

\begin{proposition}\label{prop:presque_degenere}
  There exists $\epsilon_0, C > 0$, depending only on $\Delta$, such that, for all $\epsilon \in (0,\epsilon_0)$, if $S$ is an $\epsilon$-almost-degenerate segment, then there exists a $C\epsilon$-almost-degenerate face $F$ adjacent to $S$. Furthermore for any $C\epsilon$-almost-degenerate face $F$, the three edges of $F$ are the short sub-segments of $C^2\epsilon$-almost-degenerate segments.
\end{proposition}
\begin{proof}
  Let $S$ be an almost degenerate segment and let $v_1,v_2$ be the endpoints of its small sub-segment
. By construction the segment
  $[v_1,v_2]$ is an edge of some face $F$. Since all faces are similar
  to $\Delta$, if one of the edges of $F$ is small then its area is
  small, with a ratio $C$ depending only on $\Delta$. For the same
  reason, if $F$ is almost degenerate than the length of its edges is
  small. However by proposition
  \ref{prop:prop_geometrique} point
  (\ref{prop:prop_geometrique:non_degenere}), segments have length
  bounded away from zero. Thus an edge of a small face cannot be a
  full segment which proves the end of the proposition.
\end{proof}

\subsection{The random walk and the main theorem}\label{sec:marche_aleatoire}

In this section we define the random walk and we state precisely
the invariance principle.


\begin{definition}
The random walk $X(t)$ on a non-degenerate graph $T$ is the continuous time Markov process on vertices of $T$ defined by the following jump rates. If the process is at a vertex $v$
of $T$, call $v^+, v^-$ the endpoints of the unique edge in the interior of which $v$ is contained. Then, the rates of the jumps from $v$ to $v^\pm$ are $1/\norm{v^\pm-v}$.
\end{definition}
Note that this random walk 
is automatically
a martingale thanks to the choice of the jump rates.

We can now state our main result:
\begin{theorem}[Quenched central limit theorem]\label{thm:principal}
  Let $\Delta$ be a triangle of area one and let $\lambda$ be a generic point in $S^1$. Let $X_t$ denote the random walk on $T_{\lambda \Delta}$, started from a point $v$. Then we have
\begin{itemize}
  \item $(\frac{X_{ nt}}{\sqrt{n}})_{t \in \bbR_+}$ converges in law to a Brownian motion
  \item The asymptotic covariance is proportional to the identity and does not depend on $\lambda$ or $v$ (it may depend on $\Delta$).
\end{itemize}
\end{theorem}

For the initial problem of approximation of continuous harmonic function by discrete harmonic ones, as we said in the introduction, because of the lack of speed of convergence, we do not get precise estimates but we do get convergence, for example:
\begin{corollary}[Dirichlet problem]
  Let $U$ denote a smooth open domain in $\bbR^2$ and let $f$ be an
  harmonic function on $U$ that extends continuously on the boundary
  $\partial U$. Let $T$ a non-degenerate T-graph and let $T_n$ be
  $T$ rescaled by $1/n$ (its edges are $O(1/n)$). Let $U_n = U
  \cap T_n$ and let $\partial U_n$ denote the set of vertices adjacent
  to $U_n$ and not in $U_n$. Consider $f_n$ the solution to the
  Dirichlet problem
\begin{itemize}
  \item $f_n$ is discrete harmonic in $U_n$;
  \item $f_n =f$ on $\partial U_n$ (for points outside of $U$ take the value of the nearest point in $U$).
\end{itemize}
The sequence $f_n$ converges pointwise towards $f$ as $n$ goes to infinity.
\end{corollary}
\begin{proof}
  Let $v_n$ a sequence of vertices of $U_n$ such that $v_n \rightarrow
  v_\infty$, a point in $U$. Let $X^{(n)}_t$ denote the random walk
  started in $v_n$ and let $B_t$ denote the brownian motion started in
  $v_\infty$, with the asymptotic covariance. Let $\tau_n$ denote the
  first exist time of $X^{(n)}$ from $U_n$ and let $\tau_\infty$
  denote the first exit time of $B$. We have by harmonicity $f_n(v_n)
  = \bbE[ f(X_{\tau_n}^{(n)})]$ and
  $f(v_\infty)=\bbE[f(B_{\tau_\infty})]$. Furthermore the trajectory $(B_t)_{t\ge0}$ is almost surely a continuity point of $\tau_\infty$ seen as a function of the trajectory, so convergence of $X^{(n)}$ to $B$ implies convergence in law of $X^{(n)}_{\tau_n}$ to $B_{\tau_\infty}$ and thus of $f_n(v_n)$ to $f(v_\infty)$.
\end{proof}

\begin{remark}
  In \cite{Sznitman2002}, the uniform ellipticity of the walk is an
  important part of the proof of the CLT.  Yet for the random walk
  $X_t$, the projection 
of the increment
in the direction orthogonal to the segment
  containing the current position of the walk, is 
zero.
  In this sense, the walk lacks uniform ellipticity everywhere.
  However, if one looks at the position of the walk after some finite
  time (say $1$), ellipticity is recovered. More precisely (see
  Proposition \ref{prop:ellipticite_uniforme}), the 
increments of the
discrete time
  random walk $ (X_n)_{n \in \bbN}$ have
strictly positive conditional variance in any direction, uniformly in
the current position. For this reason we will often look at the position of the random walk at integers times in the rest of the proof.
\end{remark}

\begin{definition}\label{def:chemin_oriente}
  An \emph{oriented path} on $T$ is a sequence $\psi(v_1), \psi(v_2),
  \ldots$ with $v_i \in \cH^*$ such that, for all $i$, $\psi(v_i)$ is in the interior of a non degenerate segment and $\psi(v_{i+1})$
  is one of the endpoints of this segment
  (using point   (\ref{prop:prop_geometrique:non_intersection}) of Proposition  \ref{prop:prop_geometrique}, $\psi(v_i)$ is in the
  interior of a unique segment).
\end{definition}

\begin{proposition}\label{prop:ellipticite_uniforme}
  There exists $\epsilon > 0$ (depending continuously on $\Delta$) such that, for any ${\bf n}\in\mathbb S^1$ and any $x_0\in T$, writing $(X_t)_{t \geq 0}$ the random walk started at $x_0$,
  \[
   1/\epsilon> \text{Var} (X_1 \cdot{\bf n}) > \epsilon
  \]
\end{proposition}

\begin{wrapfigure}{R}{0.5\textwidth}
\begin{center}
  \includegraphics[width=0.45\textwidth]{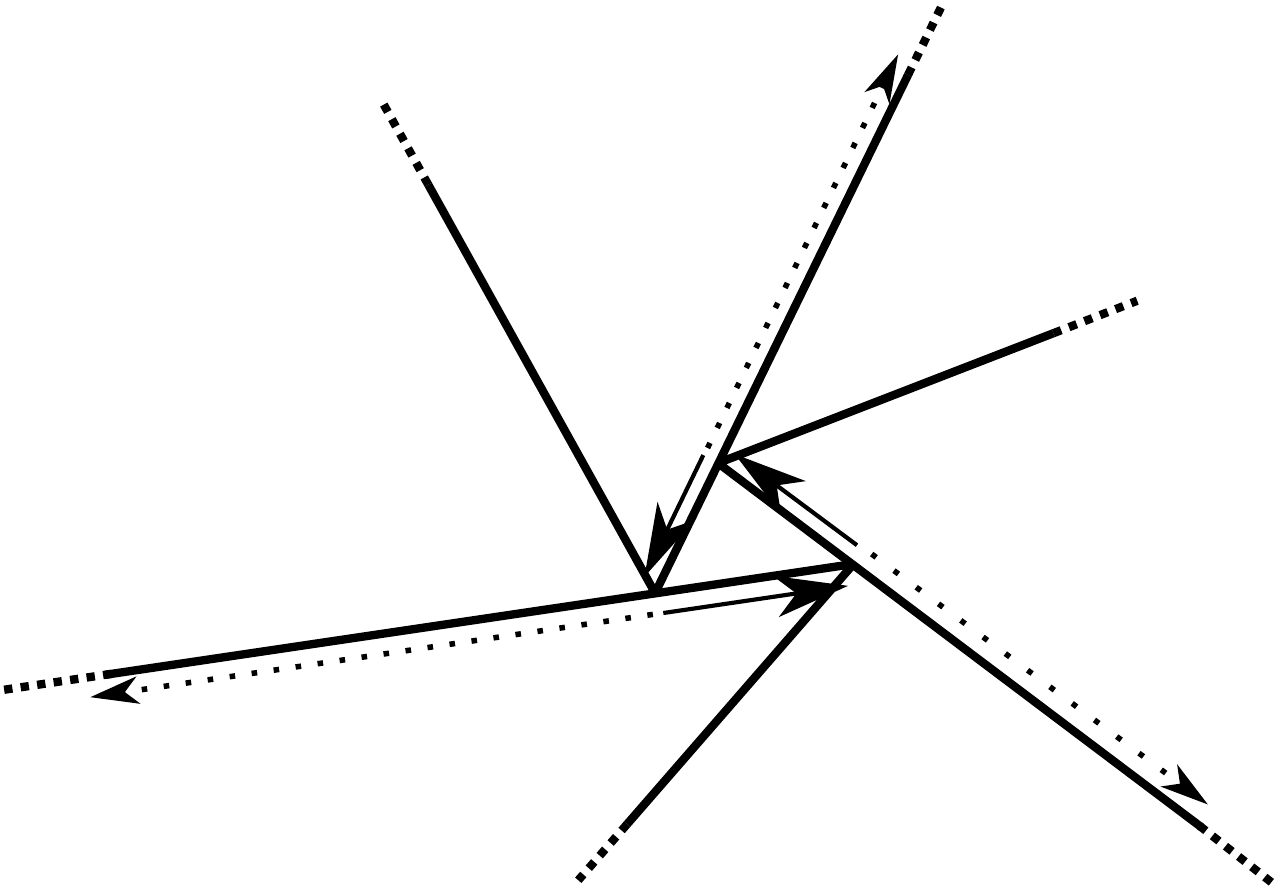}
\captionsetup{width=0.45\textwidth}
\caption{
The neighborhood of an $\epsilon$ almost degenerate face. The arrows give the possible transition of the random walk: 
the arrows with full line have high rate of order $1/\epsilon$ while the
 dotted arrows have rates of order one.}
\label{fig:presque_degenere}
\end{center}
\vspace{-0.5cm}
\end{wrapfigure} 

\noindent \emph{Proof.} 
  Since the length of segments is lower bounded (proposition
  \ref{prop:prop_geometrique} point
  (\ref{prop:prop_geometrique:face_noire})), there is a uniform lower
  bound on the probability to see at least $k$ jumps in the interval
  $(0,1)$.

Also because the length of segments is bounded, there are only two cases where the variance after a single jump can be lower than $\epsilon$: either the current position is in a segment whose direction is very close to ${\bf n}^\bot$ (so that changes of $x \cdot \bf n$ are small), or the current position is in an almost degenerate segment (so that change of $x \cdot \bf n$ are almost deterministic and small).

In the second case of an almost degenerate segment, the random walk can make either a very
small step (with high rate) or a large one (with rate of order
one). However making a small jumps brings the walk to another vertex
of an almost degenerate face $F$ where the situation is the same. Actually the
random walk is trapped inside the vertices of $F$ as long as it does
only small jumps (see figure \ref{fig:presque_degenere}). Since for
each of this vertices a large jump can occur with rate of order one, a
large jump will occur before time $1$ with probability bounded away
from zero. Finally this large jump can happen in either of the three
segments (to which  the edges of $F$ belong) with  probabilities  bounded
away from zero since the
distribution of the time spent in each of the vertices of $F$ depends
only on the ratio of the edge lengths, which are fixed. At least two of
these segments have direction far from ${\bf n}^\bot$ so jumps along
them contribute a finite amount to the variance.

In the first case of a segment with ``bad'' direction, it is clear from the construction
that the angles between neighboring segments are given by the angles
of $\Delta$. Thus neighboring segments have direction far away from
${\bf n}^\bot$. Events with two jumps in the interval $(0,1)$ have
finite probability and we just showed that they correspond to large
change of $X \cdot{\bf n}$ (unless the neighboring segments are
almost degenerate where we are back to the above case).


The upper bound is essentially trivial. Small jumps only occur inside the ``traps'' made by almost degenerate faces so they have a small effect on the position of $X_1$. As for large jumps, their rate is bounded so the probability to make a large number of large jumps is exponentially decreasing. Overall the
increment $X_1-X_0$ has exponential tails and thus a bounded variance.
\qed

\begin{proposition}\label{prop:construction_chemin_oriente}
  Let $T$ be a non-degenerate graph. For any ${\bf n} \in \mathbb S^1$
  and any vertex $x_0$, there exist two infinite oriented path $P^+$
  and $P^-$ starting in $x_0$ such that the couple $({\bf n}\cdot x,
  {\bf n}^\bot\cdot x )$ is increasing (resp. decreasing) along $P^+$
  (resp. $P^-$) for the lexicographic order (i.e. at each step either the first coordinate is increasing, or it is constant and the second coordinates is increasing). Furthermore ${\bf n}\cdot
  x$ is increasing (resp. decreasing) at least once every two steps.

  There exists $\epsilon > 0$ such that for any $x_0$ and any ${\bf n}
  \in \mathbb S^1$, there exists infinite oriented paths $\tilde P^+ =
  (x^+_k)_k$ (resp. $\tilde P^-=(x^-_k)_k$), such that, for all $k$,
  $(x^+_{k+4} -x^+_k )\cdot {\bf n} > \epsilon$ (resp. 
$(x^-_{k+4} -x^-_k)
  \cdot {\bf n} < -\epsilon$).
\end{proposition}
\begin{proof}
  By symmetry we will only construct the path $P^+$ and $\tilde P^+$. For the first one, by construction $x_0$ is in the interior of a segment $S_0$ and the point $x_1$ has to be one of the endpoints of $S$. If $S$ is not orthogonal to ${\bf n}$ then moving to one of its endpoints increases ${\bf n}\cdot x$. If $S$ is orthogonal to ${\bf n}$, we can keep ${\bf n}\cdot x$ constant and increase ${\bf n}^\bot \cdot x$.
  Finally by construction the angles between neighboring segments of $T$ are the angles of $\Delta$ so they are never aligned and two neighboring segments cannot be both orthogonal to ${\bf n}$.

  To construct $\tilde P^+$ we recall the proof of proposition
  \ref{prop:ellipticite_uniforme}. The only case where we cannot
  increase $x \cdot {\bf n}$ by $\epsilon$ in one step are almost
  degenerate segments and segments with direction close to ${\bf
    n}^\bot$. In the latter case, one step is enough to arrive to a
  ``good'' segment or an almost degenerate one. In the former case, we
  see that after at most three small steps we can find a segment where
  we can increase $x \cdot {\bf n}$ by a bounded amount. Overall after
  $4$ steps we can increase $x \cdot {\bf n}$ by $\epsilon$.
\end{proof}

\begin{remark}
  For generic $\Delta$ there is no need to distinguish the paths $P$ and $\tilde P$, the only problem comes from triangles with one right angle and two irrational ones.
\end{remark}

\section{Central limit theorem}\label{sec:TCL}

Here we give the proof of the central limit theorem for the oriented
random walk on the $T$-graph. The identification of the limiting
covariance will be achieved in Section \ref{sec:covariance} with a
completely different set of ideas: it follows from the knowledge of a
specific harmonic function on the graph, which comes from the study of
dimer models \cite{Kenyon2007}.

The proof of the CLT follows quite closely a proof in the case of
random walk in balanced random environment \cite{Lawler1982,Sznitman2002}. It is at first intriguing that we
can use arguments from random walk on random environment in our
  environment that, 
  once the single parameter $\lambda $ is fixed, is deterministic and
  quasi-periodic. However the specific structure of T-graphs will
  allow us to use an ergodicity and a compactness argument that are
  the core of the proof for a random environment.  One can also argue
  that we lose the deterministic nature of the graph in the theorem
  since it applies only to generic $\lambda$ and we have no explicit
  condition on $\lambda$.

\subsection{Periodic case}\label{sec:periodique}

If $\beta/\gamma$ and $\beta/\alpha$ are both roots of the unity then
$\phi^*$ is periodic and thus $T$ is a periodic graph. Even though
that case could be dealt with using the same techniques as the general
one, 
the result can be proved with much simpler arguments.

Let $T$ be a periodic $T$ graph and let $T_1$ be one fundamental
domain of $T$. Given the first point, there is a bijection between
random walk trajectories on $T$ and $T_1$.

The random walk on $T_1$ is a finite state Markov chain and it is not
difficult using proposition \ref{prop:construction_chemin_oriente} to
show that it has only one recurrent set. Let us assume for simplicity
that the starting point $x_0$ is recurrent. Let $\tau_0=0,\tau_1,
\ldots$ be the return times in $x_0$ of the random walk on $T_1$;
thanks to exponential mixing the $\tau_{i+1}-\tau_i$ are iid and have
exponential moments. Going back to the random walk on $T$, the
$X_{\tau_{i+1}}-X_{\tau_i}$ are also iid with exponential moments and
the central limit theorem for random walk gives the result.

\subsection{Topology on graphs}\label{sec:topologie}

The proof of the central limit theorem consists of two crucial steps. First we construct, using a compactness argument, an invariant measure for the random walk (more precisely for the environment from the point of view of the particle, see notation \ref{not:def_point_de_vue_particule} and lemma \ref{lemme:def_Q} for precise statements) and then we apply Birkhoff ergodic theorem to get convergence of the variance of the random walk. 

In this section we set up the compactness argument by defining a distance on T-graphs and giving the properties of the induced topology we will need.

\begin{definition}
  A pointed $T$-graph is a couple $T^\bullet=(T, v)$ where $T= T_{\lambda\Delta}$ for a certain $\Delta$ and $\lambda$ is a $T$-graph and $v$ is a vertex of $T$. We let $\cT^\bullet$ denote the set of pointed $T$-graphs.
\end{definition}

\begin{notation}
  We will need an explicit correspondence between black vertices of $\cH$
  and vertices of $T$. Recall that the image of a black face $b$ of
  $\cH^*$ is a segment where the image of the three vertices of the
  face are the two endpoints and a third  point on the segment (proposition \ref{prop:prop_geometrique} point (\ref{prop:prop_geometrique:endpoints}) ). We define $v(b)$ as
  this third point. Generically $v(b)$ is in the interior of the
  segment but for exceptional values of $\lambda$,  when the $\psi$-images of
  two vertices of $\cH^*$ are equal,  $v(b)$ can be one of the endpoints. Except in the non generic graph
  where this happens, $v$ is a bijection and we define coordinates on
  vertices of $T$ by using the coordinates of $v^{-1}$. We write
  $v(m,n)=v(b(m,n))$ to simplify notations.
\end{notation}

\begin{notation}
  Let $\cP$ the set of parameters $\cP= \{ (\lambda, \Delta) | \lambda
  \in \bbS^1, \Delta \text{ area one triangle}\}$. We define the mapping $\cC : \cP \rightarrow \cT^\bullet$ by
\[
  \cC(\Delta, \lambda) = (T_{\lambda\Delta}, v(0,0)) ).
\]
\end{notation}


\begin{notation}
  Let $d_H$ denote the Hausdorff distance on closed set of $\bbC$. We define $d$ on $(\cT^\bullet)^2$ by 
\[
  d( (T,v), (T', v')  ) = \inf \{ \epsilon | d_H( t_{-v} T \cup \bar B(\tfrac{1}{\epsilon}), t_{-v'} T' \cup \bar B(\tfrac{1}{\epsilon})  ) < \epsilon \}
\]
where $t_x$ denotes the translation by $x$ and $\bar B(r)$ is the closed ball of center $0$ and radius $r$.
\end{notation}

\begin{proposition}
  $d$ is a pseudo-distance on $\cT^\bullet$.
\end{proposition}

\begin{remark}
  The distance is chosen in order to measure how similar are the neighborhoods of the pointed vertex in different graphs. It is more or less the local distance of Benjamini-Schramm.
\end{remark}

\begin{notation}\label{not:def_point_de_vue_particule}
  We let $\cT^\bullet /d$ denote the quotient of $\cT^\bullet$ by $d$. We extend trivially functions into $\cT^\bullet$ as function into $\cT^\bullet /d$ without changing notation.
\end{notation}

We can now restate Proposition \ref{prop:translation}:
\begin{proposition}\label{prop:translation2}
  Let $\tau_{m,n}$ the translation defined by $\tau_{m,n} ( T, v(m',n') ) = (T, v(m+m',n+n'))$. We have 
\[ 
d\left( \tau_{m,n} \circ \cC(\Delta, \lambda) , \cC( \Delta, (\tfrac{\beta}{\gamma})^m (\tfrac{\beta}{\alpha})^n \lambda) \right) =0,
\]
which means that in the set $\cT^\bullet /d$ a translation is a special case of changing $\lambda$.
\end{proposition}

\begin{proposition}\label{prop:homeomorphisme}
  $\cC$ is continuous and onto from $\cP$ to $\cT^\bullet /d$.
\end{proposition}

\noindent\emph{Proof.}
First we prove that $\cC$ is onto. By proposition \ref{prop:translation2}, any pointed graph $(T,v)$ where $v$ is the image of a black vertex is identified in $\cT^\bullet /d$ with a graph pointed in $0$, which is in the image of $\cC$ by construction. Thus we only have to check that any vertex of a $T$-graph is of the form $v(b)$ for a certain black vertex $b$. 

\begin{wrapfigure}{R}{6.5cm}
\begin{center}
  \includegraphics[width=5cm]{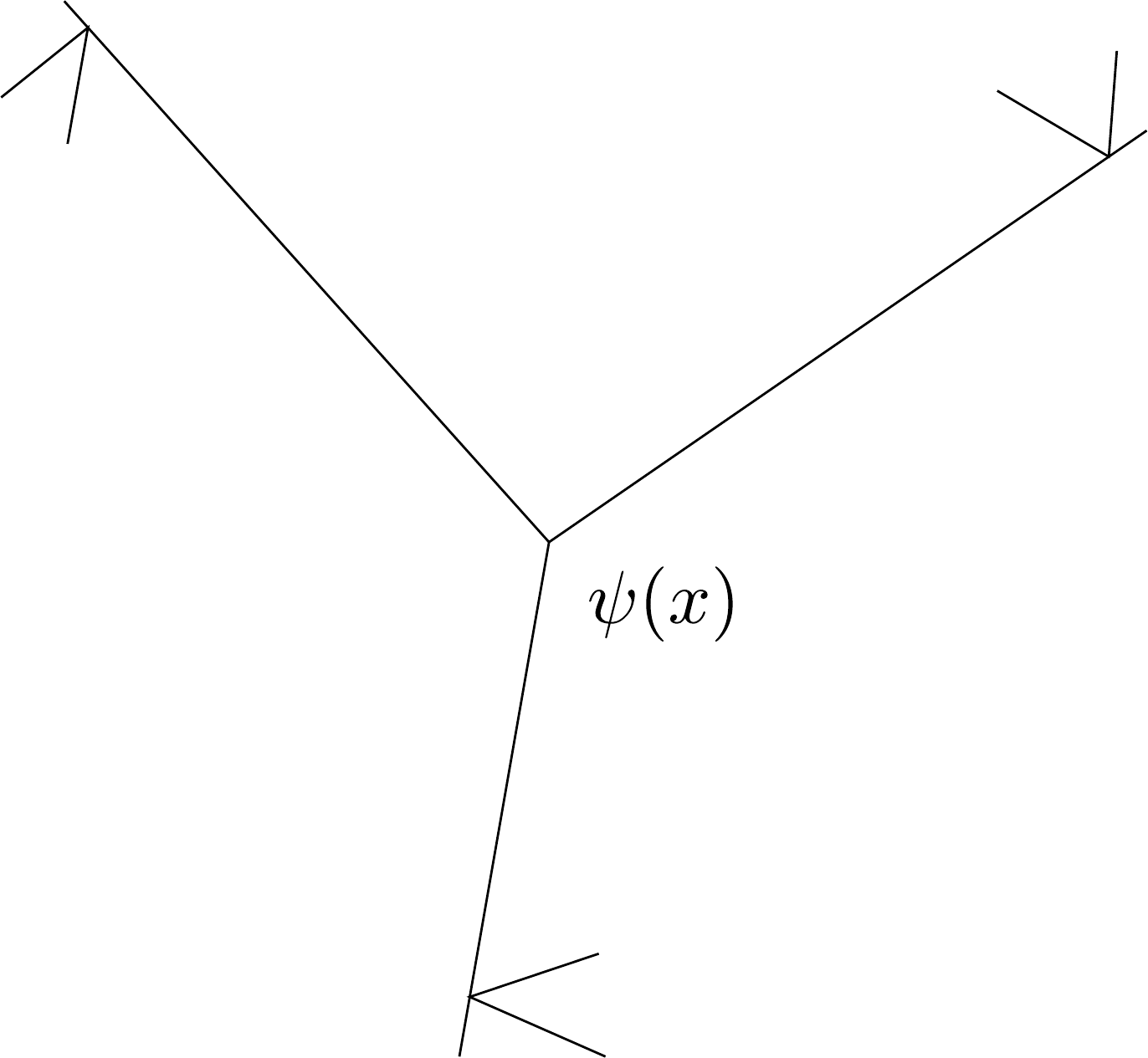}
\captionsetup{width=6cm}
\caption{The local configuration around $\psi(x)$ with a generic $\lambda$ close to $\lambda_0$. This configuration is impossible by proposition \ref{prop:prop_geometrique} point (\ref{prop:prop_geometrique:endpoints}).}
\label{fig:surjectivite}
\end{center}
\vspace{-0.3cm}
\end{wrapfigure} 

Fix a graph $T$, in proposition \ref{prop:prop_geometrique} point (\ref{prop:prop_geometrique:endpoints}) we see that a vertex of $T$ is always an endpoint of (at least) a segment so it is the image of a vertex $u$ of $\cH^*$. However this does \emph{not} imply directly that it is of the form $v(b)$. For generic $\lambda$ each vertex is in the interior of a single segment and thus is trivially of the form $v(b)$. For non generic $\lambda$ some segments will be degenerate. If the image of a certain $b$ presents such a case, then $v(b)$ is by definition the ``double'' endpoint so there is no issue with the definition, yet we have to check that no vertex of $T$ is a ``simple'' endpoint of all the segments it is in. This could be seen from the explicit definition but we give a perturbative argument (see figure \ref{fig:surjectivite} for an illustration). Suppose that for some $\lambda_0$ and some $x \in \cH^*$, $\psi(x)$ is a simple endpoint 
of the three segments corresponding to the three black 
vertices around $x$. For $\lambda$ close to $\lambda_0$ those segments will be almost degenerate with the interior point far away from $\psi(x)$ and thus $\psi(x)$ will not be of the form $v(b)$ which contradicts the results for generic $\lambda$.

  It is also clear from the explicit formula that $\phi^*$ is
  continuous as a function of $\Delta, \lambda$ and thus for fixed $x
  \in \cH^*$, $\psi(x)$ is continuous as a function of $\Delta,
  \lambda$. To finish the proof we need to show that we only need to
  know $\psi$ on a finite number of points to know $T$ in a given
  ball. This is immediate because we see from Proposition
  \ref{prop:prop_geometrique:face_blanche} that the scaling factor of
  adjacent faces cannot be both vanishingly small so there is a finite number of face in any ball.
\qed

\smallskip


\begin{corollary}\label{cor:compacite}
  $\cT^\bullet /d$ is locally compact. Let $\cT^\bullet_\Delta = \{ T_{\lambda \Delta} | \lambda \in \bbS^1\}= \cC(\bbS^1,\delta)$;  $\cT^\bullet_\Delta /d$ is compact.
\end{corollary}

\begin{remark}
  The fact that $\cT^\bullet /d $ is not compact comes from the very flat triangles. We can recover the compactness easily, for example by adding a condition on the perimeter being bounded (recall that we already fixed the area to be $1$).
\end{remark}

\subsection{Core of the proof}

In this section we give the proof of the central limit theorem through compactness and ergodicity arguments. Two key lemmas on absolute continuity of the invariant measure for the environment from the point of view of the particle will be left for the next sections.

\begin{notation}
Let $\Delta$ a triangle with angles that are not all rational
multiples of $\pi$. Let $\lambda$ such that no triangle has size $0$
in $T_{\lambda\Delta}$ (any generic $\lambda$ works). As we said
before, we could work with rational multiples of $\pi$ but our choice will make our statements about ergodicity simpler.

In this section we will only work with the set $\cT^\bullet /d$ so functions and measures will always be defined on this set.

Let $\bbU$ be the uniform measure on the unit circle and let $\bbP= \cC(\Delta, \bbU)$ be its image by the construction. It is clear from Proposition \ref{prop:translation2} and ergodicity of irrational rotations that $\bbP$ is invariant and ergodic for the group of translations $\{ \tau_{mn} |m,n \in \bbN\}$. $\bbP$ will replace the law of the environment in the proof of \cite{Sznitman2002}.
\end{notation}

\begin{notation}
We define the environment from the point of view of the particle by
$T^\bullet_t = (T, X_t)$ where $X_t$ is the random walk on $T$. We let
$W$ be its generator: $W$ is the operator on functions on $\cT^\bullet/d$ defined by $Wf(T,v) = \frac{f(T,v^+)}{\norm{v^+-v}} + \frac{f(T,v^-)}{\norm{v^--v}} -f(T,v)$. We also define $p_t$ the transition probabilities for the environment from the point of view of the particle, $p_t (T, v) f(T,v) = \bbE_{v}[f(T,X_t)]$, with the expectation with respect to the random walk on $T$ started at $v$.
\end{notation}

The main point of the proof will be to construct an invariant ergodic measure for $p_t$ and to show that it is absolutely continuous with respect to $\bbP$. This is done through approximation of the aperiodic graph $T_{\lambda\Delta}$ by periodic graphs.

\begin{notation}\label{not:def_Tn}
Let $\Delta_n$ be a sequence of triangles such that $\Delta_n \rightarrow \Delta$ and all the angles of the $\Delta_n$ are rational multiple of $\pi$. Let $\lambda_n \rightarrow \lambda$ such that none of the $T_{\Delta_n \lambda_n}$ has a face of size $0$. By construction the $T_{\Delta_n \lambda_n}$ are periodic graphs. Let $\bbP_n$ the uniform probability measure on $\{ (T_{\Delta_n \lambda_n}, v), v \in T_{\Delta_n \lambda_n} \}/d$ (which is finite by periodicity) and let $\bbQ_n$ be an invariant ergodic measure for the random walk on the same set. $\bbQ_n$ exists by general theorems on finite state Markov chains and $\bbQ_n$ is clearly $p_t$ invariant.
\end{notation}

\begin{lemma}\label{lemme:Q_est_borne_l2}
  Let $q_n = \frac{\td \bbQ_n}{\td \bbP_n}$, the $q_n$ are uniformly bounded in the $L^2( \bbP_n)$ norm. 
\end{lemma}
The proof will be given in the next section \ref{sec:mesure_invariante}.

\begin{lemma}\label{lemme:def_Q}
  There exists a measure $\bbQ$ on $\cT^\bullet_\Delta$ which is $p_t$ invariant and absolutely continuous with respect to $\bbP$.
\end{lemma}
\begin{proof}
Let $K$ be a uniform (in $n$) upper bound on the perimeter of the $\Delta_n$ and let $\cT^\bullet_K$ the set of pointed T-graphs with triangles of perimeter less than $K$. By corollary \ref{cor:compacite}, $\bbQ_n$ is a sequence of probability measures on the compact set $\cT^\bullet_K/d$ so, up to extraction, it converges towards a measure $\bbQ$. It is clear that $\bbQ$ is supported on $\cT^\bullet_\Delta$ because the parameters of the triangle are continuous functions of the graph.

Now we have to verify that $\bbQ$ is $p_t$ invariant for all $t$. $p_t$, the transition kernel of the environment from the point of view of the particle, is by definition an operator on measurable function of pointed $T$-graphs. Furthermore the jump rates of the random walk and the translations (by bounded amount) are continuous functions of the graph so $p_t$ maps continuous functions to continuous functions. Thus, for any continuous bounded function $f$, in the equality $ \E_{\bbQ_n} [ p_t f( X) ] = \E_{\bbQ_n} [ f(X) ]$ both sides go to the limit by convergence in law of $\bbQ_n$ to $\bbQ$ and:
\[
  \forall g : \cT^\bullet \rightarrow \bbR \text{ continuous bounded}, \quad \E_{\bbQ} [ p_t g(X) ] = \E_\bbQ [ g(X)]
\]
This equality means by definition that the measure $ \bbQ p_t$ and $\bbQ$ are identical on bounded continuous functions and this implies $\bbQ p_t= \bbQ$.

Finally we check that $\bbQ$ is absolutely continuous with respect to $\bbP$. It is easy to note that $\bbP_n$ converges to $\bbP$. Let $q_n = \frac{\td \bbQ_n}{\td \bbP_n}$ and let $g$ be a continuous bounded function, we have
\begin{align*}
  \abs{ \int g \td \bbQ } & = \lim \abs{ \int g \td \bbQ_n  } \text{  by convergence in law} \\
	  & = \lim \abs{ \int g q_n \td \bbP_n  } \\
	  &  \leq \limsup \left( \int \abs{g}^2 \td \bbP_n   \right)^{1/2} \left( \int \abs{q_n}^2 \td \bbP_n   \right)^{1/2}\\
	  & \leq C \norm{ g }_{L^2(\bbP)}
\end{align*}
and thus $\bbQ$ is absolutely continuous with respect to $\bbP$.
\end{proof}

\begin{lemma}\label{lemme:ergodicite_Q}
  $\bbQ$ is unique and $\bbP \sim \bbQ$. Furthermore the stationary measure on trajectories of the environment from the point of view of the particle is ergodic for the semi-group of time shifts.
\end{lemma}
The proof is essentially identical to \cite{Sznitman2002} and will be given in section \ref{sec:ergodicite}.

\begin{theorem}\label{thm:CLT_hexagone}
  Let $X_t$ denote the continuous time random walk on $T_{\lambda\Delta}$ started on $b(0,0)$. For generic $\lambda$, there exist a positive definite symmetric matrix $M$ such that $X_{Nt}/\sqrt{N}$ converges in law to the two dimensional brownian motion of covariance $M$. Furthermore $M$ does not depend on $\lambda$ (it may depend on $\Delta$).
\end{theorem}
  Remark that for any direction ${\bf n}$, $X_t.{\bf n}$ is a square integrable martingale. By the central limit theorem it is asymptotically gaussian so we only have to prove that the variance grows linearly to obtain our result. This is done by the ergodic theorem.

\begin{theorem}[Birkhoff ergodic theorem]
  Let $(\cT,\mu)$ a measured space and $F : \cT \rightarrow \cT$ a measure preserving transformation. We assume that $\mu$ is finite and $F$ invariant and ergodic, then for all $g \in L^1(\mu)$ and $\mu$ almost all $x$ we have :
\[ 
  \frac{1}{n} \sum_1^n g \circ F^k(x) \rightarrow \int g \td \mu.
\]
\end{theorem}
\begin{proof}[Proof of theorem \ref{thm:CLT_hexagone}]
  Fix ${\bf n} \in \bbS^1$ a direction and let us prove a one
  dimensional invariance principle for the random walk $X_t \cdot {\bf n}$. Since proposition \ref{prop:ellipticite_uniforme} was given with an unit time increment, we will work with the discrete time walk $(X_n)_{n \in \bbN}$. It is clear this is sufficient to get a result in the original continuous time model. Indeed the probability for the random walk to go far away from $X_n$ in the time interval $[n, n+1]$ is exponentially decreasing.

 Let $g(T,v)= \E_v[(X_1\cdot {\bf n}-X_0\cdot {\bf n})^2]$ with the expectation taken with respect to the random walk on $T$ started in $v$. By the Markov property, $g$ gives the conditional variance of any increment, more precisely
\[
  \forall k,\, \E[ (X_{k+1}\cdot {\bf n}-X_k\cdot {\bf n})^2| \cF_k] = g(T,X_k) .
\]

Consider now the set of infinite oriented paths of pointed T-graphs
(i.e. of environments viewed from the point of view of the
particle). On this set, put the measure obtained sampling the
environment (pointed T-graph) at time zero using $\bbQ$. The time shift is a measurable transformation on this set and by lemma \ref{lemme:ergodicite_Q} the measure is invariant and ergodic. The function $g$ extends trivially to a function on trajectory and is bounded so we can apply Birkhoff ergodic theorem to get
 \[
  \frac1N \sum_{k=1}^N g(T,X_k) =  \int g(T^\bullet) \td \bbQ(T^\bullet) + o(1)
 \]
where the equality holds for $\bbQ$ almost all graphs and almost all trajectories. Since $\bbQ \sim \bbP$, it is also valid for $\bbP$ almost all graphs.

The left hand side can be rewritten
\[
  \frac{1}{N} \sum_k \E[ (X_{k+1}\cdot {\bf n} - X_k\cdot {\bf n})^2| \cF_k];
\]
the right hand side is deterministic so by taking expectation on both sides we get, for $\bbP$ almost all graph,
\[
 \frac{1}{N} \E[(X_{N}\cdot {\bf n} - X_0\cdot {\bf n})^2] = \int g(T^\bullet) \td \bbQ(T^\bullet) + o(1).
\]
Remark that the limit is given by some fixed integral and does not depend on the starting point or $\lambda$.

Finally the invariance principle for martingales applies because $X_k
\cdot {\bf n}$ has $L^2$ increments and we just proved that its variance grows linearly so we have that $(\frac{X_{\lfloor Nt\rfloor}}{\sqrt N} \cdot {\bf n})_{t \geq 0}$ converges to a Brownian motion (with some unknown variance). Now this is true for any direction ${\bf n}$ so by definition $X_{\lfloor Nt\rfloor}/\sqrt{N}$ converges to a two dimensional brownian motion (again with an unspecified covariance matrix). As we said above, this is enough to conclude for the original continuous time process.
\end{proof}

\subsection{$L^2$ estimates of invariant measure}\label{sec:mesure_invariante}

In this section we prove lemma \ref{lemme:Q_est_borne_l2}. The proof
is very similar to the one in \cite{Sznitman2002} (with the notable
exception  that there they work with an underlying graph $\bbZ^2$) and is included here for the sake of completeness. This proof is slightly different from the one in \cite{Lawler1982} and it uses the approach of \cite{Kuo1990} (see Theorem 2.1 there).

We write the proof as a sequence of two lemmas. In the first one the structure of the graph appears so, since T-graphs are very different from $\bbZ^2$, we give a detailed proof. In the second one, on the other hand, the structure of the underlying does not appear so we only give a basic idea of the proof which is completely identical to \cite{Sznitman2002}.


%

\begin{notation}
Recall the notation \ref{not:def_Tn} and write $T^{(n)} = T_{\lambda_n \Delta n}$. The $T^{(n)}$ is a sequence of periodic non-degenerate $T$-graphs with parameters converging to some $(\lambda, \Delta)$ such that $T_{\lambda \Delta}$ is aperiodic and non degenerate. We assume here that the period of the $T^{(n)}$ are of order $n$ in both directions. We let $T^{(n)}_1$ denote the fundamental domain of $T^{(n)}$, seen as a finite graph embedded on the plane.
\end{notation}

\begin{lemma}\label{lemme:borne_linf_dirichlet}
   Let $X^x_t$ denote the random walk on $T^{(n)}$ started in $x$ and let $\nu$ be the time of the first exist of $T^{(n)}_1$, and for a function $f$ on $T^{(n)}_1$ let $Q f (x) = \E \sum_{0\leq k < \nu} f(X^x_k) $. We have
    \[
      \norm{ Qf }_\infty \leq C n^2 \norm{f}_{L^2(\bbP_n)}.
    \]
\end{lemma}
\begin{proof}
  We write $Q f=u$ and we drop the superscript $n$ to simplify notations. We let $\delta T_1$ denote the neighbours of $T_1$ (in the periodic graph). The first exit of $T_1$ is by definition the hitting time of $\delta T_1$. Remark that for all $x\in T_1 $, $\E[ u(X^x_1) - u(x) ] = -f(x)$ (we define by convention $u = 0$ on $\delta T_1$). 

  Let $s (x) = \{ {\bf v}  \in \bbR^2 | \forall x' \in T_1 \cup \delta T_1, u(x') \leq u(x) + {\bf v} \cdot (x'-x) \}$ and let $S = \cup_{x \in T_1} s(x)$. We start by giving a lower bound on the volume of $S$.

 Let $D$ be the diameter of $T_1 \cup \delta T_1$, let ${\bf v} \in \bbR^2$ such that $\abs{{\bf v}} < \max(u)/D$ and let $x_0 \in T_1$ be a point where $\max(u)$ is attained. By definition of the diameter, for all  $x\in T_1 \cup \delta T_1$,
\[
  u(x_0) + {\bf v}\cdot(x-x_0) > 0.
\]
Thus the function $x \rightarrow u(x_0) + {\bf v}\cdot(x-x_0) - u(x)$ is strictly positive on $\delta T_1$ (recall $u(x)=0$ on $\delta T_1$) while its minimum is negative or zero so it reach its minimum in a certain $x' \in T_1$. We see immediately that ${\bf v} \in s(x')$ and so ${\bf v} \in S$.
We just proved $\{ {\bf v} \in \bbR^2 \text{ s.t. } \abs{{\bf v}} < \max(u)/D \} \subset S$ so $S$ has a volume at least $\max(u)^2/D^2$. 

Now we will upper bound the volume of $S$ by giving an upper bound on the volume of each $s(x)$. Let $x \in T_1$, ${\bf v} \in s(x)$ and $x'$ such that $\bbP(X^x_1=x')=p > 0$. 
Since ${\bf v} \in s(x)$, the random variable $u(x) - u(X^x_1) + {\bf v}\cdot(X_1 -x)$ is positive and thus
\[
  \E[ u(x) - u(X^x_1) + {\bf v}\cdot(X_1 -x) ] \geq p\left(u(x) - u(x') + {\bf v}\cdot(x' -x)\right).
\]
The walk is balanced
$
  \E[ {\bf v}\cdot(X^x_1-x) ] =0 
$
and by definition $\E[u(X_1^x) - u(x)]=-f(x)$ so we can rewrite
\[
  {\bf v}\cdot(x' -x) \leq u(x')-u(x) + f(x)/p .
\]
We also have by applying directly the definition of $s(x)$ to $x'$ :
\[
  {\bf v}\cdot(x' -x) \geq u(x')-u(x).
\]
Finally, by uniform ellipticity we have a lower bound on $p$ so each $s(x)$ has volume at most $Cf^2(x)$. Since we already found a subset of volume $(\norm{u}_\infty/D)^2$ we get the inequality :
\[
  \norm{u}_\infty \leq C D( \sum f^2(x) )^{\frac12} \leq C' n^2 ( \tfrac{1}{\abs{T_1}} \sum f^2(x))^{\frac12}
\]
which proves the lemma.
\end{proof}

\begin{lemma}\cite{Sznitman2002}\label{lemme:borne_linf_periodique} 
  Let $X_t^x$ denote the random walk on $T^{(n)}$ started at $x$ and let $\tau$ be a geometric time of mean $n^2$ independent of the walk. We have for any function $f$ on $T_1^{(n)}$ (lifted as a periodic function on $T^{(n)}$):
\begin{equation*}
  \sup_{x \in T^{(n)}} \E[ f(X_\tau^x) ] \leq c n^2 \norm{f}_{L^2(\bbP_n)}.
\end{equation*}
\end{lemma}
This lemma is about going from ``Dirichlet boundary conditions'' to ``periodic boundary conditions''. The main idea is to introduce iterates of the stopping time $\inf \{ t >0 \text{ s.t } \norm{X_t - X_0} \geq n \}$ and to use lemma \ref{lemme:borne_linf_dirichlet} between each time.

Finally for the proof of Lemma \ref{lemme:Q_est_borne_l2}, we first see that lemma \ref{lemme:borne_linf_periodique} implies the same kind of bound for the expectation with respect to $\bbQ_n$. Then by duality we get the bound we wanted on $\norm{\frac{d\bbQ_n}{d \bbP_n}}_{L^2(\bbP_n)}$.

\subsection{Ergodicity of $\bbQ$}\label{sec:ergodicite}

In \cite{Sznitman2002} it is proved, for a random walk in ergodic
random environment on $\bbZ^d$, that if there exist a invariant
measure $\bbQ$ for the environment seen by the particle, absolutely
continuous with respect to the law of the environment $\bbP$ , then:
\begin{itemize}
  \item $\bbQ \sim \bbP$
  \item $\bbQ$ is unique
  \item the stationary random walk with initial law $\bbQ$ is ergodic (for the time shifts semi-group).
\end{itemize}

The proof translates almost identically to our setting once we have lemma \ref{lemme:accessibilite} (which was trivial in the $\bbZ^d$ case). However we will still give the proof of the first point to emphasize where we need lemma \ref{lemme:accessibilite} and also why we do \emph{not} need the graph translations to form a group.

\begin{lemma}\label{lemme:accessibilite}
  Let $T$ denote a non-degenerate $T$-graph and let $v, v'$ be two of its vertices. There exists an oriented path going from $v$ to $v'$.
\end{lemma}

\begin{wrapfigure}{R}{6.5cm}
\begin{center}
  \includegraphics[width=5cm]{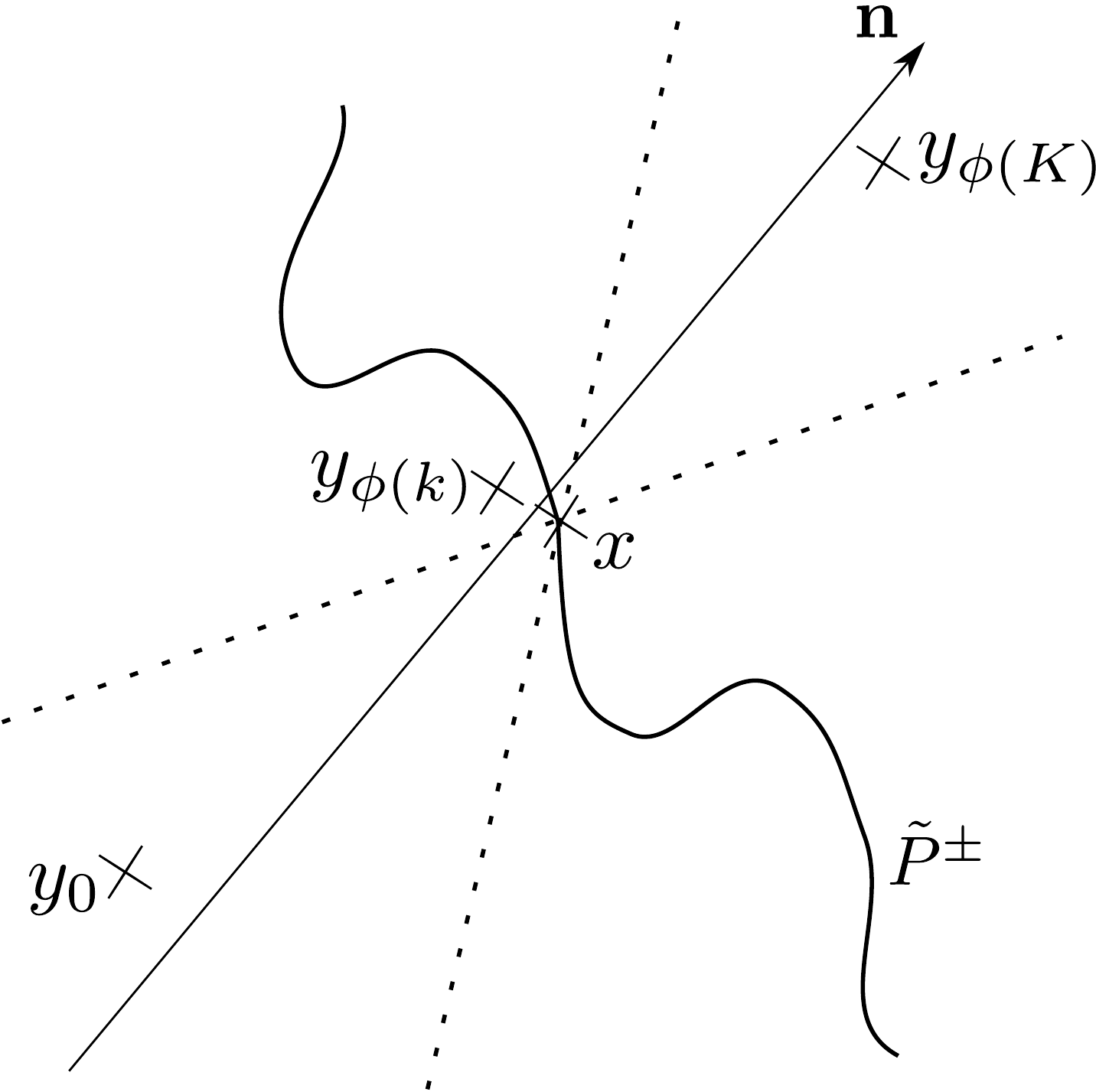}
\captionsetup{width=6cm}
\caption{An illustration of the proof that there is no infinite connected component in $T\setminus T_v$. Points $y_0$, $y_{\phi(k)}$ and $y_{\phi(K)}$ are in $T\setminus T_v$ and all in a direction close to ${\bf n}$. $x$ is a point in $T_v$ close to $y_{\phi(k)}$. The path $\tilde P^\pm$, which is known to stay inside the cone delimited by doted lines, separates $y_0$ from $y_{\phi(K)}$}
\label{fig:comp_connexe_inf}
\end{center}
\vspace{-0.5cm}
\end{wrapfigure} 

\noindent\emph{Proof.}
  Let $T_v$ be the set of points accessible by some oriented path starting in $v$. In this proof we emphasize that we will not only work with connections by oriented paths but also with connections by any non necessarily oriented path. We will use the term ``connected'' and associated definition of simple connectedness and connected component only for the latter, i.e. the usual definition when $T$ is seen as a non oriented graph.


  First we prove that $T_v$ is simply connected. Indeed if it is not
  the case let $G$ denote a finite connected component of its
  complement. Remark that any edge connecting $G$ to $T_v$ is oriented
  from $G$ to $T_v$. Let $y$ be a vertex of $G$. By the properties of
  $T$-graph, there exist exactly two vertices $y^-_1$ and $y^-_2$ that
  can be its predecessor in an oriented path and by definition of $G$
  both are in $G$. By going through all vertices of $G$ this way we
  count each edge with both ends in $G$ exactly once so we have
  $\abs{\{\text{edges of $G$}\} } = 2\abs{\{\text{vertices of $G$}\}
  }$. However we can also count edges of $G$ by looking at their
  starting point. We also have two edges going out of each vertex but
  some of them lead to vertices of $T_v$ so $2\abs{\{\text{vertices of
      $G$}\}} =\abs{\{\text{edges from $G$ to $T_v$}\}
  }+\abs{\{\text{edges of $G$}\} }$. Finally by proposition \ref{prop:construction_chemin_oriente} there are at least two edges going from $G$ to $T_v$ and we have found a 
contradiction.

  To conclude we have to show that there are no infinite connected components in the complement of $T_v$. Again by contradiction suppose there is one called $G$ and let $y_k$ be an infinite path in $G$ that stays at distance $O(1)$ of the boundary. By compactness we can extract a subsequence $y_{\phi(k)}$ such that $y_{\phi(k)}/\abs{y_{\phi(k)}}$ converges to a direction ${\bf n}$. Now remark that the paths $\tilde P^\pm$ constructed by proposition \ref{prop:construction_chemin_oriente} for the direction ${\bf n}^\bot$ have increments (every four steps) whose directions are bounded away from ${\bf n}$. In particular such a path lies completely in a cone of direction ${\bf n}^\bot$ and of angle $\pi-O(\epsilon)$, with $\epsilon > 0$ given in proposition \ref{prop:construction_chemin_oriente}. Now consider $k$ large enough and $x$ a point in $T_v$ close to $y_{\phi(k)}$ and define two paths $\tilde P^+$ and $\tilde P^-$ starting from $x$. All points of $\tilde P^+ \cup \tilde P^-$ are in $T_v$ and $\tilde P^+ \
\cup \tilde P^-$ separates the plane in two infinite connected components. By construction these connected components each include one of the connected component of the cone of direction ${\bf n}$ and angle $O(\epsilon)$. By taking $k$ large enough $y_0$ will be in one of them while $y_{\phi(K)}$ will be in the other for $K \geq k$ large enough. This is a contradiction with the fact all $y$ are in the same connected component $G$.
\qed

\begin{lemma}
  Let $\bbQ$ an invariant measure for the environment from the point of view of the particle. If $\bbQ \ll \bbP$, then $\bbQ \sim \bbP$.
\end{lemma}
\begin{proof}
  We write $f = \frac{d \bbQ}{d \bbP}$ and we let $E = \{ f=0 \}$. Recall that $p_t$ denote the probability transition function of the environment from the point of view of the particle, by construction we have $ \bbQ p_t= \bbQ$


  In particular $\bbQ p_t  1_E = \bbQ 1_E = \int 1_{ \{f=0\} } f d \bbP =0$. However we also have $ \bbQ p_t 1_E = \int p_t 1_E f d \bbP$ so $p_t 1_E =0$ on $\{ f\neq0 \} =E^c$ and thus, since $p_t 1_E \leq 1$, we get for $\bbP$ almost all pointed graph $T$:
\[
   \forall t> 0\, , 1_E(T) \geq p_t 1_E (T) = \sum_{\text{$T'$ translate of $T$}}  p_t(T \rightarrow T') 1_E(T').
\]
This implies by lemma \ref{lemme:accessibilite}
\[
  \forall T' \text{ translate of } T, 1_E(T) \geq 1_E(T')
\]
and by symmetry between $T$ and $T'$, $E$ is invariant by translations (up to a negligible set).

Now remark that this implies that $E$ is invariant for the $\tau_{mn}$ which form a group for which $\bbP$ is ergodic so we have $\bbP(E)=0$ or $1$. Since $\int f d \bbP=1$, $\bbP( E) = 1 $ is impossible.
\end{proof}

\begin{remark}\label{rq:ergodicite}
The use of the ergodic theorem here is not as straightforward as it may seem. The set of translations of the plane that send one vertex to another does \emph{not} form a group for the composition. Even worse, we cannot see a translation of the plane as a function on pointed graphs. The functions $\tau_{mn}$ on the other hand are well defined on T-graphs but are not usual translations. Indeed for fixed pointed graph $T$, $\tau_{mn} T$ is a translate of $T$ but the translation vector depends on $T$. In the ergodicity argument we need well defined functions so we have to use the $\tau_{mn}$ but the only thing we really use is the idea of a translation invariant event which does not depend on the existence of a group on the set of translation.
\end{remark}

\section{Identification of the covariance}\label{sec:covariance}

In this section we show that the covariance in the above central limit theorem is proportional to the identity. We use an approach completely different from the one above. The main idea of the proof can be summarized in the following way. We know from the connection between $T$-graph and dimer model one specific discrete harmonic function on $T$ (see \cite{Kenyon2007}). However on large scale the random walk on $T$ is similar to a brownian motion with some limit covariance matrix $M$ so discrete harmonic functions should be almost continuous harmonic function for the Laplacian associated to $M$. To identify the covariance it it thus enough to find the only Laplacian for which our specific discrete harmonic function is almost continuous harmonic.

According to the previous sketch, the first step is the construction a specific discrete harmonic function. We will actually only construct a function harmonic except for a unit discontinuity along a line, similar to $\arg(z)$.

\subsection{Dimer model}\label{sec:dimers}

We give a few background informations about the hexagonal dimer model for the reader to be able to see where our harmonic function comes from.

\begin{definition}
  A \emph{dimer covering} or perfect matching of $\cH$ is a subset $D$ of edges of $\cH$ such that each vertex is in one and only one edge of $D$. Dimer coverings of $\cH$ can also be seen as lozenge tilings of the plane.
\end{definition}


\begin{theorem}\cite{SheffieldAsterisque2005No.304}
  For all $p_a,p_b,p_c$ in $(0,1)$ such that $p_a+p_b+p_c=1$, there exists a unique ergodic Gibbs
  measure $\mu$ on dimer coverings such that :
\begin{itemize}
  \item the conditional measure on any finite subgraph of $\cH$ is uniform;
  \item vertical (resp. NE-SW, NW-SE) edges appear with probability $p_a$ (resp. $p_b,p_c$).
\end{itemize}
\end{theorem}

The distribution of dimers in these measures are given by determinantal process whose kernels are the inverses of the infinite matrix $K$ which was defined in Section \ref{sec:construction}.

\begin{theorem}\cite{Kenyon2006}
  Let $\mu$ be an ergodic Gibbs measure on dimer coverings of $\cH$. There exists an infinite matrix $K^{-1}$, indexed by white and black vertices of $\cH$ such that, for all sets of edges $(w_1b_1), \ldots ,(w_nb_n)$,
\[
  \mu( \forall i, w_ib_i \in D) = \prod_i K(w_i,b_i) \det\bigl(K^{-1}(b_k,w_l) \bigr)_{1 \leq k,l\leq n}
\]
\end{theorem}

\begin{remark}
  The notation $K^{-1}$ for the kernel is justified because it is indeed an inverse of $K$, as can be seen from the compatibility condition around single vertices. $K$ being an infinite matrix there is no contradiction with it having many inverses. However only one of them is bounded, and for this inverse we have the following expression.
\end{remark}

\begin{proposition}\cite{Kenyon2006}
  The only bounded inverse of $K$ has the asymptotic expansion 
\[
  K^{-1} \bigl(b,w\bigr) = \frac{1}{2\pi}\Im \left( \frac{\bar f(b)  g(w)}{\ell (m(w),n(w)) - \ell (m(b),n(b))} \right) + O\left(\frac{1}{\norm{w-b}^2}\right)
\]
where $O(\norm{w-b}^{-2})$ has to be understood as
$\frac{h(b,w)}{\norm{w-b}^2}$ with $h$ bounded on $\bbZ^2$ and $\Im(z)$
denotes the imaginary part of $z$. Recall that $\ell$ is an explicit linear map, $\ell(m,n) = \frac{a \alpha}{2} m - \frac{c\gamma}{2} n$.
\end{proposition}

\subsection{Covariance}

In all this section we work with a fixed graph and we will omit the parameters $\lambda, \Delta$.

\begin{definition}
  A function $h$ on $T$ is discrete harmonic if and only if, for all $x \in T$,
\[ \E_x[h(X_1) ] = h(x) \]
\end{definition}

\begin{notation}
  We define $K_T(w,b) = \Re(\bar \lambda \bar f(b)) \lambda g(w) K(b, w)$. We let $K^{-1}$ denote the only bounded inverse of $K$ defined above. It is easy to see that $K_T$ is also invertible and that the matrix $K^{-1}_T(b,w) =\frac{1}{\Re(\bar \lambda \bar f(b)) \lambda g(w)} K^{-1}(b,w)$ is an inverse of $K_T$. 
\end{notation}

\begin{remark}
  We have $K_T(w,b) = \phi(wb)$ so we have only reinterpreted a flow on edges as a matrix.
\end{remark}


Our harmonic function will be the primitive of $K^{-1}_T$.
\begin{proposition}
  Let $w$ be a face of $T$ and let $d$ be a half line from the interior of $w$ to infinity that avoids all vertices of $T$. There exists an unique (up to a constant) function $G^*_{w d} : T \rightarrow \bbC$ such that:
\begin{itemize}
  \item $G^*_{w d}$ is continuous except for $-1$ discontinuity when crossing $d$ counterclockwise.
  \item $G^*_{w d}$ is linear on edges of $T$ (on edges where it is discontinuous it is linear plus an Heaviside function)
  \item for any segment with endpoints $x^+$ and $x^-$, $G^*_{w d}(x^+) - G^*_{w d}(x^-) = K^{-1}_T(b,w) \bigl( x^+ - x^- \bigr)$ (with a additional $+1$ on discontinuous edges)
\end{itemize}
\end{proposition}
\begin{proof}
  It is clear that the properties define $G^*_{w d}$ completely, the only thing we have to check is that the definition is consistent. It is enough to check that the increments of $G^*_{w d}$ around any face sum to $0$.

  Given $w'$ a face of $T$, we write $x_1, x_2, x_3$ its vertices and $b_i$ the segment between $x_i$ and $x_{i+1}$ (with convention $x_4=x_1$), we have :
\begin{align*} 
  G^*_{w d}(x_{i+1}) - G^*_{w d}(x_i) & = K^{-1}_T( b_i,w ) (x_{i+1} - x_i) \\
	  & = K^{-1}_T( b_i,w) \phi^*( w',b_i ) \text{ by definition of $T$}\\
	  & = K^{-1}_T( b_i,w) K_T(w', b_i)
\end{align*}
on edges where $G^*_{w d}$ is continuous. On edges where $G^*_{wd}$ is discontinous the same holds with a $+1$. 

Finally $ K_T K^{-1}_T = Id$ so the above terms sum to $0$ on faces that are not $w$ since either all edges are continuous or there are exactly one $+1$ and one $-1$ discontinuity. Around the face $w$ there is a $-1$ discontinuity and the $K^{-1} K$ sum to $1$ so in the end $G^*_{w d}$ is well defined.
\end{proof}

\begin{remark}
  $G^*_{wd}$ is discrete harmonic except on edges where it is discontinuous.
\end{remark}

The asymptotic formula for $K^{-1}$ allows us to get an asymptotic expansion of $G^*_{w d}$ :
\begin{proposition}
We have, 
\[  
  G^*_{wd}(\psi(b))  = \frac{1}{2\pi} \Bigl( \arg_d(\psi(b)-w) + \frac{ \Im (\bar \lambda \bar f(w_0))}{\Re ( \bar \lambda \bar f(w_0))} \log \abs{\psi(b)- w} \Bigr) + C +O(1/(\psi(b)-w) )
\]
where $\arg_d$ denotes the determination of the argument with a $2\pi$ discontinuity on the half line $d$ and $C$ is a suitable constant.
\end{proposition}
\begin{proof}
  The proof is a direct computation. We pull back $G^*_{wd}$ as a function on $\cH^*$ where we can explicitly integrate $K^{-1}_T$ and then we use the almost linearity of the mapping $\psi$ to go back to $T$. 

Before we start with the formulas, a word about the discontinuity of $G^*$. When we consider a linear path $P_{\cH^*}$ on $\cH^*$, it corresponds to a path $P_T = \psi(P_{\cH^*})$ in $T$ which is not linear and might cross the half line $d$ a number of time. However by almost linearity we see that $P_T$ can only make a finite number of loops around $w$. Thus, taking $P_{\cH^*}$ far enough from $0$ we can make sure $P_T$ does not make any loop around $w$. For such a path the discontinuity of $G_{wd}^*$ give exactly the same contribution as the discontinuity of $\arg_d$ so we can drop it from the computation.

 Fix $b$ a black vertex of coordinates $(m,0)$, we compute $G_{wd}^*(\psi(b)) - G_{wd}^*(\psi(b(0,0)))$. For simplicity, we assume, since $G^*_{wd}$ is defined up to a constant, that $G_{wd}^*(\psi(b(0,0)))=0$. We have, writing $b_j= b(j,0)$ and $w_j= w(j,0)$: 
\begin{align*}
  G_{wd}^*( \psi(b) ) & = \sum_{j=0}^{m-1} K_T( w_j,b_j)K_T^{-1}(b_j,w_0) \\
	  & = \sum_{j=0}^{m-1} K(w_j,b_j) \Re(\bar \lambda \bar f(w_j) ) \lambda g(b_j) \frac{1}{\Re( \bar \lambda f(w_0) )} \frac{1}{\lambda g(b_j)} K^{-1}( b_j, w_0) \\
	  & = \sum_{j=0}^{m-1}  K (w_j,b_j) \frac{\Re(\bar \lambda \bar f(w_j)) }{\Re(\bar \lambda \bar f(w_0))} \Im\left( \frac{ \bar f(w_0)  g(b_j) }{2\pi \ell(j,0) }\right) + \frac{h(j)}{j^2+1}
\end{align*}
Expanding the real and imaginary parts and replacing $ K (w_j,b_j)=a$
\begin{align*}
  G_{wd}^*( \psi(b) ) & = \frac{a}{ 8i\pi\Re(\bar \lambda \bar f(w_0)) } \sum_{j=0}^{m-1} ( \bar \lambda \bar f(w_j) + \lambda f(w_j) )( \frac{ \bar f(w_0) g(b_j)}{\ell(j,0)} - \frac{ f(w_0)\bar g(b_j)}{ \bar \ell(j,0)}   ) + \frac{h(j)}{j^2+1}\\
      & = \frac{a}{ 8i\pi\Re(\bar \lambda \bar f(w_0)) } \sum_{j=0}^{m-1} \Bigl(  \frac{\bar f(w_0) g(b_j) \bar \lambda \bar f(w_j) }{\ell(j,0)}  -\frac{f(w_0) \bar g(b_j) \lambda f(w_j) }{\bar \ell(j,0)}   \\
      & \quad \quad \quad +  \frac{\bar f(w_0) g(b_j) \lambda f(w_j) }{\ell(j,0)} - \frac{ f(w_0) \bar g(b_j) \bar \lambda \bar f(w_j) }{\bar \ell(j,0)}  \Bigr) +\frac{h(j)}{j^2+1}.
\end{align*}
Thanks to the definition of $f$ and $g$, the product $\bar f(w_j) g(b_j)$ does not depend on $j$ (it is actually $\alpha$, see remark \ref{rq:valeur_fg}) so the first two terms give harmonic sums. On the other hand the two last terms have an oscillating factor so they converge and the remainder of their sum is of order $1/m$. Finally the $O(1/j^2)$ terms also converge with a $1/m$ remainder. Overall we get, for black vertices of the form $b(m,0)$ :
\begin{align*}
  G_{wd}^*( \psi(b) ) &=  \frac{a}{ 8i\pi\Re(\bar \lambda \bar f(w_0)) } \sum_{j=0}^{m-1} 2i \Im \frac{ \alpha \bar \lambda \bar f(w_0)}{ \ell(j,0) } + C + O(1/m) \\
	& = \frac{1}{2\pi} \frac{ \Im (\bar \lambda \bar f(w_0))}{\Re ( \bar \lambda \bar f(w_0) )} \log(m) +C + O(1/m)
\end{align*}
where in the last line we replaced $\ell(j,0)= a \alpha j/2$ (see section \ref{sec:construction}).

We still have to check what happens in the $n$ direction in order to identify the bounded dependence on the argument. The most natural way to do this would be to compute $G^*_{wd}$ along a circle, however for technical reason we will compute it along a parallelogram.

We first compute $G_{wd}^*( \psi(b(m,n)) )- G^*_{wd}( \psi(b(m,0) )$ for $\abs{n}\leq m$. It is also equal to a sum of $n$ terms along a straight path but this time in the $y$ direction. The computations above are still valid except that we have to replace $w_j$ and $b_j$ by the black and white vertices of the edges crossed by a path in the $y$ directions. For these NW-SE edges (recall section \ref{sec:coordonnees} and remark \ref{rq:valeur_fg}), the coordinates of $w_j$ are $(m,j)$, the coordinates of $b_j$ are $(m-1,j+1)$, $K(w_j,b_j) = c$ and $\bar f(w_j) g(b_j) = \gamma$.
Going back to the expression of $G^*_{wd}$ we get
\begin{align*}
  G_{wd}^*&( \psi(b(m,n)) )-  G^*_{wd}( \psi(b(m,0) ) \\
	&= \frac{c}{ 8i\pi\Re(\bar \lambda \bar f(w_0)) } 
	     \sum_{j=0}^{n-1} \bigl( \bar \lambda \bar f(w_j) + \lambda f(w_j) \bigr)\bigl( \frac{\bar f(w_0) g(w_j)}{\ell(m-1,j+1)} - \frac{ f(w_0)  \bar g(b_j)}{ \bar \ell(m-1,j+1)}   \bigr) + \frac{h(j)}{m^2+j^2+1} \\
	&= \frac{1}{ 8i\pi\Re(\bar \lambda \bar f(w_0)) } \sum_j 2i \Im \bigl( \frac{c\gamma  \bar \lambda \bar f(w_0)}{ \ell(m-1,j+1) }\bigr) +  \sum_j 2i \Im \bigl( \frac{c \gamma \bar \lambda f(w_0) }{\ell(m-1,j+1)} (\frac{\beta}{\gamma})^{2m} (\frac{\beta}{\alpha})^{2j} \bigr) \\
	& \quad \quad \quad +\sum_j \frac{h(j)}{m^2+j^2+1} .
\end{align*}
The last sum is of order $O(1/m)$ because it contains at most $m$ terms of order $1/m^2$, the second one is an oscillating sum of terms of order $1/m$ and is thus also $O(1/m)$, so we have :
\[
  G_{wd}^*( \psi(b(m,n)) )-  G^*_{wd}( \psi(b(m,0) ) = \frac{1}{4 \pi \Re(\bar \lambda \bar f(w_0))} \Im\Bigl(  \sum_j \frac{ \bar \lambda \bar f(w_0) c \gamma}{ -\frac{a\alpha}{2}(m-1) + \frac{c \gamma}{2} (j+1) } \Bigr) + O(\frac1m).
\]
The sum is approximately (up to $O(1/m)$) the integral of $2/z$ between $\psi(b(m,0))$ and $\psi(b(m,n))$ so it gives $2\log(\abs{\psi (b(m,n))}) +2i \arg_d(\psi( b(m,n))) - 2\log(\abs{\psi (b(m,0)})) -2i \arg_d(\psi (b(m,0)))$. Finally we have
\begin{multline*}
  \Im\bigl( \bar \lambda\bar f(w_0)\sum_j \frac{c \gamma}{ -\frac{a\alpha}{2}m + \frac{c \gamma}{2} j } \bigr) \\
      = 2\Re( \bar \lambda\bar f(w_0) ) (\arg_d(b(m,n)) - \arg_d(b(m,0))) + 2\Im( \bar \lambda \bar f(w_0)) (\log(\abs{b(m,n)}) - \log(\abs{b(m,0)}))
\end{multline*}
and together with the previous estimate on $G^*_{wd}( \psi(b(m,0) )$ we find, for any point with $\abs{n} \leq m$, 
\[
  G^*_{wd}( \psi(b(m,n) ) = \frac{1}{2\pi} \Bigl( \arg_d(\psi(b)-w) + \frac{ \Im (\bar \lambda \bar f(w_0))}{\Re ( \bar \lambda \bar f(w_0))} \log \abs{\psi(b)- w} \Bigr) + C + O(1/(\psi(b)-w) )
\]
with a constant that does not depend on $(m,n)$.

We can obtain the value of $G^*_{wd}$ on the other sides of the parallelogram $\norm{(m',n')}_\infty = m$ using exactly the same computation.
 
\end{proof}

The above proposition is already almost a proof that the covariance in the central limit theorem is proportional to the identity. Indeed the only thing left to say is that the large scale behavior of $G^*_{wd}$ has to be harmonic for the Laplacian corresponding to the limit covariance. We turn this result into a precise statement now. This requires some cumbersome integral expression but it is really only straightforward calculus.

\begin{proposition}
  The covariance matrix in theorem \ref{thm:CLT_hexagone} is
  proportional to the identity.
\end{proposition}
\begin{proof}
  Fix $v$ a vertex of $T$. To simplify notations we will assume that $v$ has coordinates $(0,0)$ in the plane where lies $T$ and is on the segment of coordinates $(0,0)$. 

  We can assume by rotating the axes that $M$ is diagonal with coefficients $M_{11} \leq M_{22}$. Fix $\epsilon > 0$ and $D$ large enough. Let $w_n$ be a sequence of faces of $T$ with $w_n-v \sim D n {\bf e_y}$. Let $d_n$ be a sequence of almost vertical half lines from $w_n$ going up and that avoid all vertices. To simplify notations, let $G^*_n$ denote $G^*_{w_n d_n}$.

Let $\tau_n$ be the minimum between $n^2$ and the first exit of $X_t$ from the ball of radius $nD/2$ and of center $(0,0)=X_0$. Let $B_t$ denote the brownian motion of covariance $M$ and let $\tau_\infty$ be the minimum between $1$ and the exit time of $B_t$ from the ball of radius $D/2$. Remark that $B_t$ is almost surely a continuity point of $\tau_\infty$, seen as a function of the trajectory $(B_t)_{t\in [0,1]}$ so $X_{\tau_n}/n$ converges in distribution to $B_{\tau_\infty}$. 
 Note also that the probability that $\tau_n \neq n^2$ is of order $e^{-D^2/8}$.

By discrete harmonicity, we have $\E_v( G^*_n(X_{\tau_n})) = G^*_n(v)$.
 On the other hand using the asymptotic formula, where we write $c_n$ for $\frac{ \Im \bar \lambda \bar f(w_n)}{\Re \bar \lambda \bar f(w_n)}$ :
\begin{align*}
  \E_v( G_n(X_{\tau_n})) & = \E_v\Biggl[ \frac{1}{2\pi} \Bigl( \arg_{d_n}(X_{\tau_n}-w_n) + c_n \log \abs{X_{\tau_n}- w_n} \Bigr) + \frac{h(X_{\tau_n})}{ \abs{X_{\tau_n}-w } } \Biggr] \\
	  &=\frac{c_n}{2\pi} \log(Dn) +  \frac{1}{2\pi} \E\Bigl[ \arg_{d_n}(\frac{X_{\tau_n}-w_n}{n}) \Bigr] +\frac{c_n}{2\pi} \E \Bigl[ \log \abs{ \frac{X_{\tau_n}- w_n }{Dn} } \Bigr] + O(\frac{1}{nD}) \\
	    & =G^*_n(v)  +\frac{c_n}{2\pi} \E \Bigl[ \log \abs{ \frac{X_{\tau_n}- w_n }{Dn} } \Bigr]  + o(1) + O(e^{-D^2/8}).
\end{align*}
In the last line, we first replaced $\tau_n$ by $n^2$ which gives an error $O(e^{-D^2/8})$ then we used the central limit theorem to replace the first expectation by $1/2+o(1)$ (remark that with our choice $\arg_{d_n}(v)=\pi$) and finally we used the asymptotic formula $G^*_n(v)  =\frac{c_n}{2\pi} \log(Dn) + 1/2 + O(\frac{1}{nD})$. To finish the proof we just have to prove that $\frac{c_n}{2\pi} \E \Bigl[ \log \abs{ \frac{X_{\tau_n}- w_n }{Dn} } \Bigr]$ does not vanish with $n$ and is bigger than $O(e^{-D^2/8})$. 

We can choose $w_n$ such that $c_n$ converges to a non zero value. For the expectation, the central limit theorem gives the limit 
\begin{align*}
  \E \Bigl[ \log \abs{ \frac{X_{\tau_n}- w_n }{Dn} } \Bigr] & \rightarrow \E[ \log \abs{\frac{B_{\tau_\infty}}{D} - i}  ].
\end{align*}
In the limit of large $D$, the integral on the right hand side becomes
\begin{align*}
  \E[ \log \abs{\frac{B_{\tau_\infty}}{D} - i}  ] & = \int \log \abs{\frac xD - i} \td N(0, C) + O(e^{-D^2/8})\\
      &= \int \tfrac{1}{2}\log\Bigl( M_{11}\frac{ x^2}{D^2} + (1-\sqrt{M_{22}}\frac{y}{D})^2 \Bigr)\frac{e^{-\frac{x^2}{2}} e^{-\frac{y^2}{2}}}{2\pi} \td x\td y + O(e^{-D^2/8})\\
      & =\tfrac{1}{2} \int \log\Bigl( 1 - 2\sqrt{M_{22}}\frac yD + M_{11}\frac{x^2}{D^2} +M_{22}\frac{y^2}{D^2}\Bigr) \frac{e^{-\frac{x^2}{2}} e^{-\frac{y^2}{2}}}{2\pi}\td x \td y+ O(e^{-D^2/8})\\
      & =\tfrac{1}{2} \int \left(- 2\sqrt{M_{22}}\frac yD + M_{11}\frac{x^2}{D^2} +M_{22}\frac{y^2}{D^2} - \tfrac{1}{2}(\sqrt{M_{22}}\frac{2y}{D})^2 \right)\frac{e^{-\frac{x^2}{2}} e^{-\frac{y^2}{2}}}{2\pi}\td x \td y + O(\frac{1}{D^3})\\
      & = \tfrac{1}{2D^2}( M_{11} - M_{22} )\int y^2 \frac{e^{ -\frac{y^2}{2} }}{\sqrt{2\pi}} \td x + O(\frac{1}{D^3}).
\end{align*}
The expansion of $\log$ is legal in the fourth line by dominated convergence. In the last line we just remark that we can separate the integrals over $x$ and $y$ and that both give the same term. For $M_{11} \neq M_{22}$ the integral is of order $1/D^2$ and we have the contradiction we were looking for.
\end{proof}

\section*{Acknowledgments}
I gratefully acknowledge the help and support of my advisor Fabio Lucio Toninelli. His contribution was invaluable at all stages of this work. I also thank Christophe Sabot for introducing me to the book \cite{Sznitman2002}.

\bibliographystyle{alpha}

\bibliography{biblioTCL_Tgraph}

\end{document}